\documentclass[12pt,a4paper]{article}
\usepackage[nottoc,numbib]{tocbibind}
\usepackage[english]{babel}
\usepackage[utf8x]{inputenc}
\usepackage{amsfonts,longtable,amssymb,amsmath,latexsym,dsfont}
\usepackage{wrapfig}
\usepackage[title]{appendix}
\usepackage{amsthm, amsmath, amssymb, amsfonts, amscd, amscd, geometry}
\newsavebox{\ssa}
\usepackage{bbm}

\newenvironment{sproof}{%
  \proof}{\endproof}

\newtheorem{thm}{Theorem}[subsection]
\newtheorem{prop}[thm]{Proposition}
\newtheorem{lemma}[thm]{Lemma}
\newtheorem{cor}[thm]{Corollary}
\newtheorem{cord}[thm]{Corollary/Definition}
\newtheorem{conj}[thm]{Conjecture}

\theoremstyle{definition}
\newtheorem{def0}[thm]{Definition}

\theoremstyle{remark}
\newtheorem{rem}{Remark}
\newtheorem{ex}[thm]{Example}

\newcommand{\lthm}{Łoś's theorem }
\newcommand{\Rep}{\textrm{Rep}}
\newcommand{\Hom}{\textrm{Hom}}
\newcommand{\End}{\textrm{End}}
\newcommand{\Id}{\textrm{Id}}

\newcommand{\Res}{\textrm{Res}}
\newcommand{\Mat}{\textrm{Mat}}
\newcommand{\Aut}{\textrm{Aut}}
\newcommand{\Ind}{\textrm{Ind}}
\newcommand{\Tr}{\textrm{Tr}}
\newcommand{\Fun}{\textrm{Fun}}
\newcommand{\mspec}{\textrm{mspec}}

\long\def\/*#1*/{}

\title{Classification of simple algebras in the Deligne category $\Rep(S_t)$}

\author{Nate Harman, Daniil Kalinov}
\date{}

\begin{document}

\maketitle

\begin{abstract}
    We classify simple associative and Lie algebras inside the Deligne categories $\Rep(S_t)$, answering a question posed by Etingof.
\end{abstract}

\tableofcontents

\section*{Introduction}

Deligne defined a family of categories $\Rep(S_t)$, where $t\in F$ is an arbitrary element of a ground field $F$ of characteristic $0$ \cite{deligne2007categorie}. These categories interpolate, in an appropriate sense, the categories of representations of symmetric groups as rigid symmetric monoidal $F$-linear categories.

In \cite{etingof2014representation} and \cite{etingof2016representation} Etingof laid out a program of ``Representation Theory in Complex Rank" which hopes to classify certain algebraic structures internal to these Deligne categories (and their relatives for other families of groups) over the complex numbers and to interpolate more complicated objects in representation theory.

One such problem proposed by Etingof was to characterize the simple algebras inside $\Rep(S_t)$ for all $t\in \mathbb{C}$, the conjectural characterization being that all such algebras should be interpolations of families of compatible simple $S_n$-algebras (that is, $\mathbb{C}$-algebras with $S_n$-equivariant multiplication, and no non-trivial $S_n$-invariant ideals). Where by algebras we could mean associative, commutative, or Lie.

The case of simple commutative algebras was studied by Sciarappa in \cite{sciarappa2015simple}. He characterized which families of simple commutative $S_n$-algebras have interpolations to the Deligne categories $\Rep(S_t)$, and showed that if $t$ is transcendental, these form a complete list of the isomorphism classes of simple commutative algebras in $\Rep(S_t)$. However, his methods are unable to rule out the possibility of certain ``exotic" simple commutative algebras existing at special algebraic values of $t$.

In \cite{harman2016deligne} the first author gave a new interpretation for the Deligne categories $\Rep(S_t)$ at algebraic values of $t$ in terms of model theory and the modular representation theory of symmetric groups, and using this sketched a proof that indeed Sciarappa's classification of simple commutative algebras holds at algebraic values of $t$ too.

\medskip

In this paper we complete this program by classifying the simple associative and Lie algebras inside $\Rep(S_t)$ for all values of $t\notin \mathbb Z_{\ge 0}$. The paper is structured as follows.

\medskip

The first two sections are somewhat auxiliary. In the first section we review the definition and properties of the Deligne category $\Rep(S_t)$, with a focus on the model theoretic ultrafilter interpretation introduced in \cite{harman2016deligne}. The second section contains some technical lemmas about subgroups of symmetric groups, $S_n$-algebras, and ultraproducts thereof.

\medskip

In section 3 we give the classification of associative algebras in $\Rep(S_t)$. Theorem \ref{assoc} says these are classified by  a choice of a subgroup $H$ of some symmetric group $S_j$, a projective representation of $H$, and an object of $\Rep(S_{t-j})$. In particular, these exactly correspond to families of simple associative $S_N$-algebras over $\mathbb{C}$ which have an interpolation in $\Rep(S_t)$.  As in the commutative case, there are no ``exotic" simple algebras appearing for special values of $t$.

\medskip

In section 4 we focus on the case of Lie algebras, and here the situation is somewhat more complicated. Given an object $V$ of dimension $0$ in a rigid symmetric tensor category, one can construct a Lie algebra object $\mathfrak{psl}(V)$ of dimension $-2$ as a subquotient of $V\otimes V^*$.  By combining this construction with induction from subgroups we see that there are many examples of simple Lie algebra objects in $\Rep(S_t)$ for $t$ algebraic which do not come from interpolating simple $S_n$-Lie algebras over $\mathbb{C}$. However using the ultraproduct interpretation we can view these as interpolating $S_n$-Lie algebras in large positive characteristic. Theorem \ref{lieclass} says that this is all that can happen, and every simple Lie algebra object in $\Rep(S_t)$ indeed comes from either  interpolating simple $S_n$-Lie algebras over $\mathbb{C}$, or by interpolating these algebras built from algebras of the form $\mathfrak{psl}(V)$ in positive characteristic.

\subsection*{Acknowledgement} 

The second author wants to thank Pavel Etingof for suggesting this project to him  and for the valuable discussions they had about it. The first author is partially supported by NSF postdoctoral fellowship award number 1703942.

\section{Deligne category $\Rep(S_t)$.}

\subsection{Preliminary notations and definitions.}

First we need to establish some notations which we will frequently use.
\begin{def0}
 By $\textbf{\Rep}(S_n)$ denote the category of finite-dimensional representations of $S_n$\footnote{Symmetric group of rank $n$.} over $\overline{\mathbb Q}$. By $\textbf{\Rep}_{p}(S_n)$ denote the category of finite-dimensional representations of $S_n$ over $\overline{\mathbb F}_p$.
\end{def0}

Note that for $p > n$ the latter category is semi-simple and the irreducible objects are the same as in characteristic $0$. More precisely, the irreducible representations over $\overline{\mathbb F}_p$ can be obtained as a reduction modulo $p$ of irreducible $\overline{\mathbb Q}$-representations which sit inside irreducible $\mathbb C$-representations as a $\overline{\mathbb{Q}}$-lattice of full rank.  Below we will always work in the positive characteristic with $p > n$.

\begin{def0}
 By $\mathcal A_n$ denote the subgroup of $S_n$ consisting of even permutations.
\end{def0}

\begin{def0}
 For a Young diagram $\lambda$, by $l(\lambda)$ denote the number of rows of the diagram and by $|\lambda|$ the number of boxes.
\end{def0}

The irreducible objects of both of these categories are in 1-1 correspondence with the  Young diagrams of weight $n$. So let us make the following definitions.

\begin{def0}
 By $X(\lambda)$ denote the irreducible representation of $S_{|\lambda|}$ corresponding to the Young diagram $\lambda$. The field over which this representation is defined will be evident from the context.
\end{def0}
\begin{def0}
 Let $X_k$ denote the vector space $F^k$ with a structure of the representation of $S_k$ given by interchanging the basis vectors. Here also the field $F$ will be evident from the context.
\end{def0}

Below we will frequently use the following operation on Young diagrams:
\begin{def0}
 For a Young diagram $\lambda$ and an integer $n \ge \lambda_1+|\lambda|$ denote by $\lambda|_n$ the Young diagram $(n - |\lambda|, \lambda_1, \dots, \lambda_{l(\lambda)})$.
\end{def0}

We will also need the following definition in the next subsection:

\begin{def0}
 Denote by $FP_{n,m}$ a vector space over a field $F$ with a basis given by all possible partitions of an $n+m$-element set. Graphically an element of the basis is represented by the two rows of $\bullet$'s, the first of length $n$ and the second of length $m$, where all $\bullet$'s belonging to the same part of the partition are connected by edges. So, in other words, it is a graph on $n+m$ vertices, the set of connected components of which corresponds to a partition of $n+m$ (It doesn't matter which graph with the fixed set of the connected components to choose).
 
 Also define a map $\phi_t^{n,m,k}: FP_{m,k} \times FP_{n,m} \to FP_{n,k}$ for $t \in F$ as follows.
 Consider $\lambda \in FP_{n,m}$ and $\mu \in FP_{m,k}$. Take a vertical concatenation of graphical representations of the corresponding partitions (the last one on top) and identify the rows of length $m$. After this we are left with a partition of three rows of $\bullet$'s of length $n,m$ and $k$. Now let's denote by $l(\mu,\lambda)$ the number of connected components consisting purely of $\bullet$'s lying in the second row. Also consider a partition of rows $n,k$ consisting of the same connected components as the partition of rows $n,m,k$ but with elements of the second row deleted, denote it by $\mu \cdot \lambda$. Then $\phi_t^{n,m,k}(\mu, \lambda) = t^{l(\mu,\lambda)}\mu \cdot \lambda$.
 
 Define $FP_{n}(t)$ to be $FP_{n,n}$ with a structure of algebra given by the map $\phi_t^{n,n,n}$. This algebra is called the partition algebra and it was introduced by Purdon in \cite{purdon1991potts}.
\end{def0}

\subsection{Definition and properties of $\Rep(S_t)$.}

From now one $t \in \mathbb C$ is any number.

Here we will briefly discuss the definition of $\Rep(S_t)$ and state some important properties it enjoys. For more about this see \cite{deligne2007categorie}, \cite{comes2009blocks}, \cite{etingof2014representation}.\footnote{For the general theory of tensor categories see \cite{etingof2016tensor}.}

First we need to define a preliminary skeletal category $\Rep_0(S_t)$:
\begin{def0}
  $\Rep_0(S_t)$ is a skeletal tensor category\footnote{Note that the notion of tensor category slightly differs from \cite{etingof2016tensor} since we do not require our categories to be abelian, nevertheless all categories we are going to work with will be abelian}. Its objects are elements of $\mathbb Z_{\ge 0}$, which can be graphically represented by rows of $\bullet$'s, and denoted by $[n]$.
  
  The set of morphisms $\Hom_{\Rep_0(S_t)}([n],[m])$ is equal to $\mathbb CP_{n,m}$ and the composition maps are given by $\phi_t^{n,m,k}$.
  
  Tensor product on objects is defined by the horizontal concatenation of rows and on morphisms by the horizontal concatenation of diagrams. All objects $[n]$ are self-dual.
\end{def0}

Now we can define the Deligne category $\Rep(S_t)$ itself:
\begin{def0}
 The Deligne category $\Rep(S_t)$ is a Karoubian envelope of the additive envelope of $\Rep_0(S_t)$. It is a rigid symmetric tensor category.
\end{def0}

\begin{def0}
 An object $[1]$ is called the fundamental representation and is denoted by $\mathcal X$.
  An object $[0]$ is called the trivial representation and is denoted by $\mathbb C$ (by a slight abuse of notation).
\end{def0}

The important properties of $\Rep(S_t)$ are listed below:
\begin{prop}
\textbf{a)} For $t \notin \mathbb Z_{\ge 0}$ $\Rep(S_t)$ is a semisimple category. \\
\textbf{b)} For $t \notin \mathbb Z_{\ge 0}$ the simple objects of $\Rep(S_t)$ are in 1-1 correspondence with Young diagrams of arbitrary size. They are denoted by $\mathcal X(\lambda)$. Moreover $\mathcal X(\lambda)$ is a direct summand in $[|\lambda|]$. \\ 
\textbf{c)} The categorical dimension of $\mathcal X$ is $t$ and of $\mathbb C$ is $1$. \\
\textbf{d)} All $\mathcal X(\lambda)$ are self-dual.
\end{prop}

Also we have an important universal property of the Deligne category:
\begin{prop}
(8.3 in \cite{deligne2007categorie})
For any $\mathbb C$-linear symmetric tensor category $\mathcal T$, the category of $\mathbb C$-linear symmetric tensor functors from $\Rep(S_t)$ to $\mathcal T$ is equivalent to the category $\mathcal T^f_t$ of commutative Frobenius algebras in $\mathcal T$ of dimension $t$. The equivalence is obtained using the functor which sends a functor $F$ to an object $F(\mathcal X)$.
\end{prop}
 
 Here by commutative Frobenius algebra in symmetric monoidal category we mean an object $T$ with the  following structure. First it is an associative commutative algebra with structure given by $\mu_T,1_T$. Second this object is rigid (dual objects exist). And finally if we define a map:
 $$
 \Tr: T \xrightarrow{1 \otimes coev_T} T \otimes T \otimes T^* \xrightarrow{\mu_T \otimes 1} T \otimes T^* \xrightarrow{ev_T} 1 \ ,
 $$
 then the pairing $T \otimes T \xrightarrow{\mu_T} T \xrightarrow{\Tr} 1$ is non-degenerate, i.e. corresponds to isomorphism between $T$ and $T^*$ under identification $\Hom(T\otimes T,1) = \Hom(T, T^*)$.
 
\subsection{Ultrafilters and ultraproducts.}

It will be important for us that we can think of $\Rep(S_t)$ for  $t \notin \mathbb Z_{\ge 0}$ as a limit of categories of representations of $S_n$ over $\overline{\mathbb F}_p$, when $n,p \mapsto \infty$. To formalize this statement we will need to introduce ultrafilters and ultraproducts below.

We will quickly define what ultrafilters and ultraproducts are, state their main properties and give some examples. The following discussion is taken from \cite{kalinov2018finite}. For more details see \cite{schoutens2010use}.

\begin{def0}
An ultrafilter $\mathcal F$ on a set $X$ is a subset of $2^{X}$ satisfying the following properties:

$\bullet$ $X \in \mathcal F$ ;

$\bullet$ If $A \in \mathcal F$ and $A \subset B$, then $B \in \mathcal F$ ;

$\bullet$ If $A,B \in \mathcal F$, then $A\cap B \in \mathcal F$ ;

$\bullet$ For any $A\subset X$ either $A$ or $X \backslash A$ belongs to $A$, but not both.
\end{def0}

There is an obvious family of examples of ultrafilters: $\mathcal F_x = \{ A| x \in A \}$ for $x \in X$. Such ultrafilters are called principal. Using Zorn's lemma one can show that there exist  non-principal ultrafilters $\mathcal F$ if the cardinality of $X$ is infinite. Also it follows that all cofinite sets belong to such an $\mathcal F$ (but not all  sets belonging to $\mathcal F$ are cofinite). From now on we will denote by $\mathcal F$ a fixed non-principal ultrafilter on $\mathbb N$. Also by something being true for ``almost all $n$", we will mean that it is true for all $n$ in some $A \in \mathcal F$. Note that by definition of an ultrafilter, if two statements hold for almost all $n$, then their conjunction holds for almost all $n$. Also note that if  for almost all $n$ the disjunction of a finite number of statements holds, then one of them holds for almost all $n$ (if not then each of them holds on some subset $A \notin \mathcal F$ and the union of this subsets is not in $\mathcal F$). We will use these elementary observations quite frequently. 

Let's now define a notion of an ultraproduct. 
\begin{def0}
Suppose we have a collection of sets $S_i$ labeled by natural numbers. Suppose that for almost all $x\in A$ one has $S_x \ne \emptyset$. Then $\prod_{\mathcal F}S_x$ is the quotient of $\prod_{x \in A} S_x$ by the following relation: $\{s_x\} \sim \{s'_x\}$ iff $s_x = s_x'$ for almost all $x$.
If for almost all $x$ one has $S_x = \emptyset$, then $\prod_{\mathcal F}S_x = \emptyset$. 
The set $\prod_{\mathcal F}S_x$ is called the ultraproduct of $S_x$. 
\end{def0}
Usually we will denote $\{ s_x \} \in \prod_{\mathcal F}S_i$ by $\prod_\mathcal F s_x$.

First, let's note that ultraproduct inherits any relation or operation, which was defined for almost all $n$. Indeed to apply it to the elements of ultraproduct we can just apply it to the corresponding sequences of elements, and if it was an operation get a sequence of elements, or if it was a relation get a sequence of Boolean values, which are going to be the same for almost all $n$, since there are finite number of values.

The most important property of ultraproducts is the following:
\begin{thm}\textbf{\lthm} (Theorem 2.3.2 in \cite{schoutens2010use})

Suppose we have a collection of sequences of sets $S^{(k)}_i$ for $k = 1,\dots,m$ and a collection of sequences of elements $f^{(r)}_i$ for $r = 1,\dots, l$ and a  formula of a first order language $\phi(x_1,\dots,x_l, Y_1, \dots, Y_m)$ depending on some parameters $x_i$ and sets $Y_j$. Denote by $S^{(k)} = \prod_{\mathcal F}S^{(k)}_{n}$ and $f^{(r)} = \prod_{\mathcal F} f^{(r)}_n$.  Then 
$\phi(f^{(1)}_n, \dots, f^{(l)}_n, S^{(1)}_n, \dots S^{(m)}_n)$ is true for almost all $n$ iff  $\phi(f^{(1)}, \dots, f^{(l)}, S^{(1)}, \dots S^{(m)})$ is true.
\end{thm}

In plain language this means that if we have a sequence  of collections of sets with some algebraic structure given by maps between them, then, first, we have the corresponding maps between the ultraproducts of these sets. And, second, these maps satisfy a given set of axioms or properties for the ultraproducts iff they satisfy these axioms/properties for almost all $n$. Also frequently it is useful to think about  ultraproducts as a some kind of limits as $n \mapsto \infty$.

We  give a number of examples of such constructions, which are going to be useful to us below:

\begin{ex}
 If $S_i$ is a sequence of monoids/groups/rings/fields then $\prod_{\mathcal F} S_i$ with operations given by taking the ultraproduct of the operations in the corresponding sets of $\Hom_{Sets}$ gives us a structure of monoid/group/ring/field by Łoś's theorem.
\end{ex}

\begin{ex}
 If $V_i$ are finite-dimensional vector spaces over a field $k$, then $\prod_{\mathcal F} V_i$ is not necessarily a finite-dimensional vector space, since the property of being finite-dimensional cannot be written in a first-order language. But if the dimensions of $V_i$ are bounded, then they are the same for almost all $i$ and hence $V$ has the same dimension (for example, because the ultraproduct of bases is a basis).
\end{ex}

\begin{ex}
 Take the ultraproduct of a countably infinite number of copies of $\overline{\mathbb Q}$. By \lthm $\prod_{\mathcal F} \overline{\mathbb Q}$ is a field, which is algebraically closed. It has characteristic zero since $\forall k  \in \mathbb Z$ such that $k\ne 0$ it follows that $ k = \prod_{\mathcal F} k\ne 0$.  Also it is easy to see that its cardinality is continuum. Hence by Steinitz's theorem\footnote{This theorem tells us that two uncountable algebraically closed fields are isomorphic iff their characteristic and cardinality are the same. It is proven in \cite{steinitz1910algebraische}.} $\prod_{\mathcal F} \overline{\mathbb Q} \simeq \mathbb C$. Note that there is no canonical isomorphism.
\end{ex}
 
 \begin{ex}
 Take the ultraproduct of $\overline{\mathbb F}_{p_n}$ for some sequence of distinct prime numbers $p_n$. As before, by \lthm $\prod_{\mathcal F} \overline{\mathbb F}_{p_n}$ is a field, which is algebraically closed. Also as before it has cardinality continuum. Now $k = \prod_{\mathcal F} k \ne 0$, since it is equal to zero for at most finite number of  $k$.  Hence $\prod_{\mathcal F} \overline{\mathbb F}_{p_n} \simeq \mathbb C$, again not in a canonical way.
\end{ex}

\begin{ex}
 Suppose $\mathcal C_i$ is a collection of small categories. We can define an ultraproduct category $\widehat{\mathcal C} = \prod_{\mathcal F} \mathcal C_i$ as a category with objects $Ob( \widehat{ \mathcal C}) = \prod_{\mathcal F} Ob(\mathcal C_i)$ and $\Hom_{\widehat{\mathcal C}}(\prod_{\mathcal F} X_i,\prod_{\mathcal F} Y_i) = \prod_{\mathcal F} \Hom_{\mathcal C_i}(X_i,Y_i)$; the composition maps are given by the ultraproducts of composition maps, i.e. $(\prod_{\mathcal F}f_i) \circ (\prod_{\mathcal F}g_i) = \prod_{\mathcal F} (f_i \circ g_i)$. By \lthm this data satisfies the axioms of a category. If the categories $\mathcal C_i$ have some structures, for example the structures of an abelian/monoidal/tensor category, then $\widehat{\mathcal C}$ also has these structures\footnote{But the finite-length property, for example, does not survive.}.
 
 Usually $\widehat{\mathcal C}$ is too big and it is interesting to consider some full subcategories $\mathcal C$ in there, or, equivalently, consider ultraproducts only of some sequences of objects of $\mathcal C_i$, for example bounded in some sense. 
 
 This construction obviously extends to essentially small categories up to an equivalence of such a category with respect to all relevant structures. All categories which we will consider are essentially small, and for all the questions we are going to discuss the fact that ultaproducts are defined up to equivalence does not matter, so we won't bother mentioning this later.
\end{ex}

\subsection{$\Rep(S_t)$ as an ultraproduct.}
Here, we will  show how to construct $\Rep(S_t)$ using ultraproducts, and discuss some important consequences of this construction. See \cite{deligne2007categorie},\cite{harman2016deligne}\footnote{For the similar discussion about $\Rep(GL_t)$ see \cite{deligne2007categorie}, \cite{harman2016deligne}, \cite{kalinov2018finite}.}.

We want to apply the last example of the previous section to $\mathcal C_i = \textbf{Rep}(S_{n_i}, \mathbb K_i)$ -- the tensor category of finite-dimensional representations of $S_{n_i}$ over $\mathbb K_i$. As was stated before (Definition 1.1.1) we will denote by $\textbf{Rep}(S_n) = \textbf{Rep}(S_n,\overline{\mathbb Q})$ and by $\textbf{Rep}_p(S_n) = \textbf{Rep}(S_n, \overline{\mathbb F}_p)$. We have the following result (Introduction of \cite{deligne2007categorie} or Theorem 1.1 in \cite{harman2016deligne}):

\begin{thm}

\textbf{a)} Suppose $t\in \mathbb C$ is transcendental. Consider $\widehat{\mathcal C} = \prod_{\mathcal{F}} \textbf{Rep}(S_n)$. Denote by $X_i = \overline{\mathbb Q}^i$ -- the fundamental representation of $S_i$ and $X_t = \prod_{\mathcal{F}}X_i$. Fix an isomorphism $\prod_{\mathcal F}\overline{\mathbb Q}\simeq \mathbb C$ such that $\prod_{\mathcal F} i = t$. Then the full subcategory of the $\prod_{\mathcal F}\overline{\mathbb Q}$-linear category $\widehat{\mathcal C}$ generated by $X_t$ under taking tensor products, direct sums and direct summands is equivalent to the $\mathbb C$--linear category $\Rep(S_t)$, in a way consistent with the above isomorphism $\prod_{\mathcal F}\overline{\mathbb Q} \simeq \mathbb C$.

\textbf{b)} Suppose $t \in \mathbb C$ is algebraic but not integer, with minimal polynomial $q(x) \in \mathbb Z[x]$. Fix a sequence of distinct primes $p_n$ and sequence of integers $t_n$ tending to infinity such that $q(t_n) = 0$ in $\mathbb F_{p_n}$. Moreover fix an isomorphism $\prod_{\mathcal F}\overline{\mathbb F}_{p_n}\simeq \mathbb C$ such that $\prod_{\mathcal F} t_i = t$. Set $\widehat{\mathcal C} = \prod_{\mathcal{F}} \textbf{Rep}_{p_n}(S_{t_n})$. Denote by $X_{t_i}= \overline{\mathbb F}_{p_i}^{t_i}$  the fundamental  representation of $S_{t_i}$ and set $X_t = \prod_{\mathcal{F}}X_{t_i}$. Then the  full subcategory of the $\prod_{\mathcal F}\overline{\mathbb F}_{p_n}$-linear category $\widehat{\mathcal C}$ generated by $X_t$ under taking  tensor products, direct sums and direct summands is equivalent to the $\mathbb C$-linear category $\Rep(S_t)$,  in a way consistent with the above isomorphism $\prod_{\mathcal F}\overline{\mathbb F}_{p_n}\simeq \mathbb C$.
\end{thm}
\begin{proof}
\textbf{a)} First let us prove that it is indeed possible to fix such an isomorphism. The ultraproduct $\prod_{\mathcal F}i$ is an element of $\mathbb C$. Suppose it is algebraic over $\mathbb Q$, then it should satisfy a monic equation $f$ with coefficients in $\mathbb Q$. Then by \lthm for almost all $i$ we have $f(i) = 0$, but since this is true for infinite number of distinct $i$'s, it follows that $f = 0$. Hence by contradiction we conclude that $\prod_{\mathcal F} i$ is transcendental. Now by fixing an automorphism of $\mathbb C$ over $\mathbb Q$ we may send this transcendental number to $t$.

So we have a tensor category $\widehat{\mathcal C}$ linear over $\mathbb C$, with an object $\prod_{\mathcal F}X_i$ of dimension $t$. Since every $X_i$ is a commutative Frobenius algebra, it follows by \lthm that $X_t$ is also a commutative Frobenius algebra. Hence by Proposition 1.2.5 we obtain a monoidal symmetric functor $F:\Rep(S_t) \to \widehat{\mathcal C}$. Since $\Rep(S_t)$ is generated by $\mathcal X$ under taking  tensor products, direct sums and direct summands, it follows that the image of $\Rep(S_t)$ under $F$ is  the full subcategory $\mathcal C$ in $\widehat{\mathcal C}$ generated by $X_t$ under taking  tensor products, direct sums and direct summands. So we know that $F:\Rep(S_t) \to \mathcal C$ is essentially surjective. Now it is enough to prove that it is fully faithful. 

Note that it is enough to prove that 
$$
\prod_{\mathcal F} \Hom_{S_n}(X_n^{\otimes r}, X_n^{\otimes s}) = \Hom_{\Rep(S_t)}([r], [s]) \ ,
$$
and that the composition maps are the same. Indeed both categories can be obtained as the Karoubian envelope of the additive envelope of the categories consisting of all $[s]$ or  $X_t^{\otimes r}$ respectively.

But this follows from Theorem 2.6 in \cite{comes2009blocks}. Indeed there it is stated that there is an isomorphism between $\overline{\mathbb Q}P_{r,s}$ and $\Hom_{S_{n}}(X_n^{\otimes r},X_n^{\otimes s})$ for $n>r+s$. So for almost all $n$ we have $\Hom_{S_n}(X_n^{\otimes r}, X_n^{\otimes s}) = \overline{\mathbb Q}P_{n,m}$. Also Proposition 2.8 in the same article states that under this isomorphism the composition rule on $\Hom_{S_n}(X_n^{\otimes r}, X_n^{\otimes s})$ transforms into the composition rule on $\overline{\mathbb Q}P_{n,m}$ in the definition of $\Rep_0(S_t)$. So it follows that indeed $\prod_{\mathcal F} \Hom_{S_n}(X_n^{\otimes r}, X_n^{\otimes s}) = \Hom_{\Rep(S_t)}([r], [s])$, and the composition rule is the same.

\textbf{b)} First, again, we need to explain how we can fix such an isomorphism.
Let us prove that there is indeed an infinite number of pairs $t_n$ and $p_n$ such that $q(t_n) = 0 \mod p_n$. It is enough to show that there are infinite number of primes dividing the numbers $q(n)$ (if in this case the sequence $t_n$ is bounded, it follows that some $q(t_n)$ is divisible by an infinite number of prime numbers, which is absurd). Suppose it is not so, and there are only $k$ such primes. Fix $C$ such that  we have $q(n) < C \cdot n^{\deg(q)}$ for all positive  integer $n$. Denote by $Q$ the number of integers of the form $q(n)$ for $n \in \mathbb Z_{\ge 0}$ such that $q(n)<N$. By the above inequality this number is at least $\frac{1}{C}\cdot N^{\frac{1}{\deg(q)}}$. On the other hand the number $P$ of  numbers less than $N$ divisible only by $k$ fixed primes  is less or equal to $\log_2(N)^k$, since each prime number is at least $2$. Hence for big enough $N$ we have $P<Q$, which contradicts the hypothesis\footnote{This proof is also written by the first author in his paper, see the proof of Prop. 2.2 in \cite{harman2016deligne}.}. 

So we indeed can choose such unbounded sequences $t_n$ and $p_n$. Now by \lthm it follows that $\prod_{\mathcal F}t_n$ is a root of $q$ in $\mathbb C$, so by composing with an automorphism of $\mathbb C$ we may assume that under an isomorphism $\mathbb C \simeq \prod_{\mathcal F} \overline{\mathbb F}_{p_n}$, $\prod_{\mathcal F}t_n$ maps to $t$. 

The rest of the proof is the same since the representation theory of $S_{n}$ is the same in zero characteristic and in characteristic $p>n$, and $p_n>t_n$ for almost all $n$.
\end{proof}

\begin{rem}
Generally proving something in the case of the algebraic $t$ is harder than in the case of the transcendental $t$. Thus we will for the most time think about the transcendental case as a subcase of the algebraic case using the following formalism. By $\overline{\mathbb F}_0$ we will mean $\overline{\mathbb Q}$, and so the case $t_n = n$, $p_n = 0$ gives us transcendental $t$.

Also we will always assume that the sequences $p_n$ and $t_n$ are the sequences from Theorem 1.4.1b) corresponding to the given $t$.
\end{rem}

Now we can understand how to obtain the objects $\mathcal X(\lambda)$ as ultraproducts.
Indeed $FP_n(t)$ for any $t \ne 0, \dots, 2n$ is a semisimple algebra with the same collection of idempotents given by the specialization of idempotents from the same algebra, but there we treat $t$ as a formal variable\footnote{See discussion in chapter 3.3 in \cite{comes2009blocks}.}.
So from Proposition 3.25 in \cite{comes2009blocks} it follows that the idempotent corresponds to $\mathcal X(\lambda)$ in $FP_n(t)$ is the same one which corresponds to the irreducible representation of $S_k$ with $k>2n$ given by the Young diagram $(k -|\lambda|, \lambda_1, \dots, \lambda_{l(\lambda)})$. 

So using the notation introduced in Definition 1.1.2  the following Corollary holds.
\begin{cor}
The irreducible object $X(\lambda)$ of $\Rep(S_t)$ can be obtained as an ultraproduct of irreducible objects of $\textbf{\Rep}_{p_n}(S_{t_n})$ as $\mathcal X(\lambda) = \prod_{\mathcal F}X_{t_n}(\lambda|_{t_n})$.
\end{cor}

\begin{rem}
Note that this means that if $V$ as an object of $\Rep(S_t)$ equals to the ultraproduct $V = \prod_{\mathcal F}V_n$, then almost all $V_n$ cannot contain the sign representation. Indeed, since $V = \oplus \mathcal X(\lambda_i)$, we have for almost all $n$, $V_n = \oplus X(\lambda_i|_n)$, hence the height of the Young diagrams appearing in $V_n$ is bounded for almost all $n$. But if $V_n$ would contain the sign representation for almost all $n$ it would mean that each $V_n$ would contain the diagram of height $t_n$ which contradicts  the boundness of heights of the Young diagrams.
\end{rem}

We also need to describe the generalizations of the induction and restriction functors. First let's define the latter using the universal property of $\Rep(S_t)$\footnote{In this paper we use $\boxtimes$ to denote a Deligne tensor product of locally finite abelian categories, for the definition see 1.11 in \cite{etingof2016tensor}.}.
\begin{def0}
 Consider the category $\Rep(S_{t-k}) \boxtimes \textbf{\Rep}(S_k)$ for an integer $k$, and in it the object $\mathcal X \otimes \mathbb C \oplus \mathbb C \otimes X_k$. This object is a commutative Frobenius algebra, and has dimension $t$, so by universal property we have a functor $\Rep(S_t) \to \Rep(S_{t-k}) \boxtimes \textbf{\Rep}(S_k)$. This functor is called the restriction functor and is denoted by $\Res_{S_{t-k} \times S_k}^{S_t}$. 
\end{def0}

Now we want to describe it in terms of ultraproducts.
\begin{prop}
The functor $\Res_{S_{t-k} \times S_k}^{S_t}$ is equal to $\prod_{\mathcal F}\Res_{S_{t_n-k} \times S_k}^{S_{t_n}}$, where the latter functors are regular restriction functors for finite groups.
\end{prop}
\begin{proof}
A priori this ultraproduct functor is a not a functor between $\Rep(S_t)$ and \\  $\Rep(S_{t-k})\boxtimes \textbf{\Rep}(S_k)$, but between the bigger categories. So we need to check that if we restrict it to $\Rep(S_t)$, we will indeed get objects of $\Rep(S_{t-k}) \boxtimes \textbf{\Rep}(S_k)$.

So consider $\mathcal X(\lambda) = \prod_{\mathcal F}X(\lambda|_n)$. Now $$
\left (\prod_{\mathcal F} \Res_{S_{t_n-k} \times S_k}^{S_{t_n}}\right)(\mathcal X(\lambda)) = 
\prod_{\mathcal F} \Res_{S_{t_n-k} \times S_k}^{S_{t_n}}(X(\lambda|_n)) = \prod_{\mathcal F} \bigoplus_{|\mu| = k, |\nu| = t_n-k \ } c_{\nu,\mu}^{\lambda|_n}X(\nu)\otimes X(\mu)  \ ,
$$ 
where the $c$'s are the Richardson-Littlewood coefficients. 

So if $t_n$ is sufficiently big, the gap between the first and the second rows of $\lambda|_{t_n}$ is bigger than $k$. For such $t_n$ the skew shapes $\mu / \nu$ for admissible $\nu$ are all disconnected, there is a part above the first row and the part in the first row. Note that we also can put any sequence of numbers in the part lying in the first row. Hence if we denote by $M(\mu)$ the set of weights $\mu'$ (not necessarily partitions) such that $0 \le \mu'_i\le \mu_i$ it follows that the previous expression equals  $$
 \prod_{\mathcal F} \bigoplus_{|\mu| = k, \mu' \in M(\mu), \nu} c_{\nu,\mu'}^{\lambda}X(\nu|_{t_n})\otimes X(\mu)= \bigoplus_{|\mu| = k, \mu' \in M(\mu), \nu} c_{\nu,\mu'}^{\lambda}\mathcal X(\nu)\otimes \mathcal X(\mu) \ ,
$$
where $c_{\nu,\mu'}^{\lambda} $ is the number of skew-shapes $\lambda / \nu$ of weight $\mu'$. Indeed in this formula we just first summed over the possible choices of fixing the length and the content of the first row (by fixing $\mu'$) and then the rest. So the image of $\mathcal X(\lambda)$ under the ultraproduct functor indeed lies in $\Rep(S_{t-k}) \boxtimes \textbf{\Rep}(S_k)$. 

Now since $\prod_{\mathcal F}\Res_{S_{t_n-k} \times S_k}^{S_{t_n}}(\mathcal X) = \mathcal X \otimes \mathbb C \oplus \mathbb C \otimes X_k$, it follows that the ultraproduct functor sends $\mathcal X$ to $\mathcal X \otimes \mathbb C \oplus \mathbb C \otimes X_k$. So by universality we conclude that the functors are the same.
\end{proof}

\begin{cord}
There is a functor biadjoint to $\Res_{S_{t-k} \times S_k}^{S_t}$, denoted by $\Ind_{S_{t-k} \times S_k}^{S_t}$ and it is equal to $\prod_{\mathcal F}\Ind_{S_{t_n-k} \times S_k}^{S_{t_n}}$.
\end{cord} 
\begin{proof}
It can be proven in the same way as above that the ultraproduct of induction functors defines a functor into the Deligne category. 

After we know this, by \lthm it follows that this functor is biadjoint to the restriction functor, since it is true in finite rank.
\end{proof}

Note that this allows us to define the restriction and induction functors for any subgroup of $S_k$ in the following way:
$$
\Ind_{S_{t-k} \times G}^{S_t} = \Ind_{S_{t-k} \times S_k}^{S_t} \circ \Ind_{S_{t-k}\times G}^{S_{t-k}\times S_k} \ ,
$$
where the later functor is defined to be $\Ind_{S_{t-k}\times G}^{S_{t-k}\times S_k} = \left ( \Id \boxtimes Ind_{G}^{S_k}\right)$.

The same thing holds for restrictions.

\section{Technical results on representations of $S_N$.}

In this section we prove some technical lemmas which we will use extensively in our proofs of classification. The reader can skip this section at first, and then go back when the need arises. 

\subsection{Facts about subgroups in $S_N$ of small index.}

In this subsection we will prove that under some restrictions on the index of a subgroup of $S_N$ it is conjugate to either $\mathcal{A}_n \times H$ or $S_n \times H$, where $H$ is a subgroup of $S_{N-n}$.

So suppose $N > 10$, $r$ an integer less then $N/2$, and $G$ a subgroup of $S_N$ of index less then ${ N \choose r}$. First following Theorem 5.2 in \cite{dixon1996permutation} we have the following proposition:
\begin{prop}
Under the above assumptions up to conjugation $G$ contains the group $\mathcal{A}_{N-j}$ with $j<r$, where $\mathcal{A}_{N-j}$ is the group of even permutations of the first $N-j$ elements.
\end{prop}

Now we have the second result:
\begin{prop}
Suppose $G$ is a subgroup of $S_{N}$ which contains $\mathcal{A}_{N-j}$  and $N > 2j+7$, then $G$ is conjugate to either $S_{N-j'} \times H$ or $\mathcal{A}_{N-j'} \times H$ for some $H \subset S_{j'}$ and $j' \le j$.
\end{prop}
\begin{proof}
 Let's consider the standard action of $S_{N}$ on $N$ elements. Consider the orbit of the first element under the action of $G$. By assumption it contains the first $N-j$ elements. Up to taking a group conjugate to $G$ (we conjugate by an element fixing the first $N-j$ elements) we may assume that the orbit of the first element under $G$ is equal to the first $N-j'$ elements for $j' \le j$. We want to prove that $G$ contains $\mathcal{A}_{N-j'}$. To do this, it is enough to prove that any $3$-cycle consisting of the first $N-j'$ elements belongs to $G$. 

 Let's denote by $B$ the set of first $N-j$ elements and by $C$ the set of $j-j'$ elements directly after $B$. So we need to consider $3$-cycles of four types. 
 
 The first case of a $3$-cycle consisting solely of elements of $B$ is trivial by assumption. 
 
 The second case is the $3$-cycle permuting elements $x,y,z$ such that $x,y \in B$ and $z \in C$. Also let's denote the first element by $1$. Since $z$ belongs to the orbit of $1$ under $G$, it follows that $\exists g \in G$ such that $g(z)=1$. Now since $|B| = N-j$ is bigger than $j$ by at least 6, it follows by the pigeonhole principle that there exist two elements $a,b \in B$ such that $g(a),g(b) \in B$ and all $a,b,g(a),g(b),x,y$ are distinct. Now consider a double transposition $\tau$ which interchanges $a \leftrightarrow x$ and $b \leftrightarrow y$, it belongs to $\mathcal{A}_{N-j}$ and hence to $G$. Now consider a $3$-cycle $\pi$ permuting $g(a), g(b)$ and $1$. It also belongs to $\mathcal{A}_{N-j}$ and hence $G$. Now $ \tau g^{-1}\pi g\tau \in G$ is a $3$-cycle permuting $x,y,z$. Indeed if $c \in N$ is not equal to $a,b,x,y,z$ then both $\pi$ and $\tau$ act trivially, hence $c$ maps to $c$ under the above map. The elements $a$ and $m$ first map to $x,y$ accordingly, then under $g$ they map to something on which $\pi$ acts trivially, so they are mapped back and then back to $a$ and $b$. Now $(x,y,z)$ first map to $(a,b,z)$ then to $(g(a),g(b),1)$, then to $(g(b),1,g(a))$, then to $(b,z,a)$ and to $(y,z,x)$.
 
 The third case is $x \in B$ and $y,z\in C$. Again suppose $g\in G$ maps $z$ to $1$. Again by the pigeonhole principle there are $a,b,c \in B$ such that $g(a),g(b),g(c) \in B$ and $a,b,c,g(a),g(b),g(c),x$ are distinct. By $\tau$ denote the double transposition interchanging $a \leftrightarrow x$ and $b \leftrightarrow c$, as before $\tau \in G$ by the assumptions. Now by the previous two cases a $3$-cycle $\pi$ which permutes $g(a),1,g(z)$ belongs to $G$. Hence by the same logic as above $\tau g^{-1}\pi g\tau \in G$ is the required $3$-cycle.
 
 The final case is $x,y,z \in C$. As before, fix $g\in G$ mapping $z$ to $1$. By the above cases there is a $3$-cycle $\pi \in G$ permuting $g(y),g(z),1$. Then $g^{-1}\pi g$ is the required cycle. 
 
 Hence $\mathcal{A}_{N-j'} \subset G'$, where $G'$ is a  group conjugate to $G$. Since the orbit of $1$ consists of the first $N-j'$ elements, it follows that $G' \subset S_{N-j'} \times S_{j'}$. By the above discussion we are limited to the two cases: $G' = \mathcal{A}_{N-j'} \times H$ or $G' = S_{N-j'} \times H$, where $H \subset S_{j'}$.
\end{proof}

Now we are ready to state the main theorem of this section:
\begin{thm}
Suppose $G \subset S_N$ has index less than ${N \choose r}$ for  $N > 2r + 8$. Then $G$ is conjugate either to $S_{N-j} \times H$ or $\mathcal{A}_{N-j} \times H$ for some $H \subset S_{j}$ and $j \le r$.
\end{thm}
\begin{proof}
Using Proposition 2.1.1 we conclude that the conjugate group $G'$ contains $\mathcal{A}_{N-j'}$ for $j'<r$. Now using the Proposition 2.1.2 we conclude that the conjugate group $G''$ is equal to either $S_{N-j} \times H$ or $\mathcal{A}_{N-j} \times H$ for some $H \subset S_{j}$ and $j \le j' \le r$, since $N > 2r+7 \ge 2j'+7$.
\end{proof}

\subsection{Lemmas on ultraproducts of representations of $S_{t_n}$.}

\begin{lemma}
Suppose $V$ is an object of $\Rep(S_t)$ such that $V = \prod_{\mathcal F}V_n$ and $V_n = \Ind^{S_{t_n}}_{G_n}(W_n)$ for some subgroup $G_n \subset S_{t_n}$. Then it follows that $G_n = S_{t_n-j}\times H$ for some $j \in \mathbb Z_{> 0}$ and $H\subset S_j$, for almost all $n$. Also $W = \prod_{\mathcal F}W_n$ is an object of $\Rep(S_{t-j}) \boxtimes \Rep(H)$, hence $V = \Ind^{S_t}_{S_{t-j}\times H}(W)$.
\end{lemma}
\begin{proof}
Suppose $V$ is equal to the sum of $l(V)$ simple objects of $\Rep(S_t)$ such that each one is a subobject of $[m]$ with $m\le m(V)$. Then for almost all $n$ we have $V_n$ being equal to the sum of $l(V)$ irreducible representations included in $V^m$ for $m \le m(V)$. Hence for almost all $n$ we have $\dim V_n \le l(V) \cdot (t_n)^{m(V)}$. But since $V_n = \Ind^{S_{t_n}}_{G_n}(W_n)$ for $G_n \subset S_{t_n}$, we know that $\dim V_n = \dim W_n \cdot |S_{t_n}| / |G_n| \ge |S_{t_n}| / |G_n|$. Hence we obtain the following inequality:
$$
l(V)\cdot (t_n)^{m(V)} \ge |S_{t_n}| / |G_n| \ .
$$
So we have a subgroup of $G_n \subset S_{t_n}$ with an index bounded by $l(V) \cdot t_n^{m(V)}$. Since ${t_n \choose m(V)+1}$ is a polynomial of degree $m(V)+1$ with the highest term being equal to $\frac{t_n^{m(V)+1}}{(m(V)+1)!}$ it follows that all but finite number of  $t_n$ we have ${t_n \choose m(V)+1} \ge l(V) \cdot (t_n)^{m(V)}$. Hence for almost all $n$, $G_n$ satisfies the condition of Theorem 2.1.3 with $N=t_n$ and $r= m(V)+1$. Thus for almost all $n$ we have, after a conjugation, $G_n' = S_{t_n-j_n} \times H_n$ or $G_n' = \mathcal{A}_{t_n-j_n} \times H_n$ for $j \le m(V) + 1$ and $H_n \subset S_{j_n}$. 

For conjugate subgroups $G$ and $G'$ the objects $\Ind^{S_N}_{G}(U)$ and $\Ind^{S_{N}}_{G'}(U)$ are isomorphic for the correct choice of the action of $G$ and $G'$ on $U$. Hence we may suppose that $V_n = \Ind_{G'_n}^{S_{t_n}}(W_n)$. But there is a finite number of subgroups $H$ in $S_j$ for $j \le m(V) + 1$, hence there is a finite number of ways to choose $G'_n$ for every $n$. Thus for almost all $n$ we have the same $j_n=j$ and $H_n=H \subset S_j$, and $G_n = S_{t_n-j} \times H $ or $G_n = \mathcal{A}_{t_n-j} \times H $ for almost all $n$. First we need to rule out the possibility $G_n=\mathcal{A}_{t_n-j} \times H$.

So suppose $G_n = \mathcal{A}_{t_n-j} \times H$. Then 
$$
V_n = \Ind_{G_n}^{S_{t_n}}(W_n) = \Ind_{S_{t_n-j}\times S_j}^{S_{t_n}}(U_n) \ ,
$$
where $U_n$ is an $S_{t_n-j}\times S_j$-module with a structure of representation of $S_{t_n-j}$ induced from the structure of representation of  $\mathcal A_{t_n-j}$. But any such representation  is equivalent to itself tensored with the sign representation, hence if a partition $\lambda$ appears in the decomposition, so does its conjugate $\lambda^*$.

However for any partition we have that $l(\lambda) \times l(\lambda^*) \ge |\lambda|$, so in particular $\max \{ l(\lambda), l(\lambda^*)\} \ge \sqrt{|\lambda|}$. Therefore any representation induced from $\mathcal{A}_{t_n-j}$ to $S_{t_n-j}$ contains an irreducible component corresponding to a partition of length at least $\sqrt{t_n - j}$.
So
$$
U_n = \sum_i X(\lambda_i)\otimes X(\mu_i) \ ,
$$
where $|\lambda_i| = t_n-j$, $|\mu_i| = j$ and at least one of $\lambda_i$ is of length at least $\sqrt{t_n-j}$.

So we have:
\begin{equation}
V_n = \bigoplus_{i,\zeta} X(\zeta)^{\oplus c^\zeta_{\lambda_i,\mu_i}} ,
\end{equation}
where $\zeta$ are partitions of $t_n$ and $c$'s are the Littlewood-Richardson coefficients. 

Suppose $\lambda_j$ is of length at least $\sqrt{t_n-j}$. Then there is $\zeta$ such that $c^{\zeta}_{\lambda_j,\mu_j} \ne 0$ and hence $\zeta$ contains $\lambda_i$ and thus $l(\zeta)\ge l(\lambda_j)\ge \sqrt{t_n-j}$. So the lengths of Young diagrams appearing in $V_n$ are unbounded. But this contradicts to $V$ being an object of Deligne category, because all the simple objects appearing in $V$ lie in $\mathcal X^{\otimes m}$ for some bounded $m$, and hence the length of the Young diagrams appearing in $V_n$ should be bounded for almost all $n$.
Hence $G_n = S_{t_n-j}\times H$.

So $V_n = \Ind^{S_{t_n}}_{S_{t_n-j}\times H}(W_n)$. The last thing to check is that $\prod_{\mathcal F}W_n$ is an object of Deligne category. Write $W_n = \oplus W^{k}_n\otimes U_k$, where $U_k$ is all possible irreducible representations of $H$. Now $\prod_{\mathcal F}W_n$ lies in $\Rep(S_{t-j})\boxtimes \Rep(H)$ iff the number of irreducible representations in the sequence $W^{k}_n$ for each $k$ is bounded and the number of boxes in the corresponding Young diagrams in all the rows except the first one is bounded too. But note that when we induce, each representation in $W^{k}_n$ gives us at least one irreducible representation in the resulting object, so if the number of irreducible representation is unbounded here it is also unbounded in $V_n$. Also if the number of boxes in all the rows except the first one is unbounded, then  it follows that the number of boxes in irreducible representations of $V_n$ is also unbounded. Indeed by Littlewood-Richardson we only add boxes to diagrams when applying induction. Hence for $V$ to lie in $\Rep(S_t)$, $W$ also should lie in $\Rep(S_{t-j})\boxtimes \Rep(H)$. So we are done and $V = \Ind^{S_{t}}_{S_{t-j}\times H}(W)$.
\end{proof}

The next lemma concerns the projective representations of $S_n$.
Denote by $\widehat{S}_n$ the double cover of $S_n$. We may regard projective representations of $S_n$ as a linear representations of $\widehat{S}_n$.  We will need the following result (\cite{kleshchev2012small}):

\begin{thm}
Suppose $n>12$ and $p\ne 2$, then any  irreducible projective representation of $S_n$, which is faithful as $\widehat{S}_n$ representation has dimension at least:
$$
\min \left(2^{\lfloor \frac{n-1-\kappa_n}{2}\rfloor}, 2^{\lfloor \frac{n-2-\kappa_{n-1}}{2}\rfloor}(n-2-\kappa_n-2\kappa_{n-1})\right) \ ,
$$
where $\kappa_n$ is $1$ if $p|n$ and $0$ otherwise.

Under the same assumptions, any irreducible projective representation of $\mathcal A_n$, which is faithful as $\widehat{\mathcal A}_n$ representation has dimension at least:
$$
\min \left(2^{\lfloor \frac{n-2-\kappa_n}{2}\rfloor}, 2^{\lfloor \frac{n-3-\kappa_{n-1}}{2}\rfloor}(n-2-\kappa_n-2\kappa_{n-1})\right) \ ,
$$
where $\kappa_n$ is the same.
\end{thm}

Since the only nontrivial normal subgroups of $\widehat{S}_n$ are $\widehat{\mathcal A}_n$ and the central subgroup,  it follows that any non-linear representation of $S_n$ satisfies the condition of the above theorem. Now we can apply this to obtain the following lemma:

\begin{lemma}
Suppose $W_n$ is a sequence of projective representations of $S_{t_n}$(or $\mathcal A_{t_n}$) for some unbounded sequence $t_n$, such that $\dim W_n \le Mt_n^L$. Then almost all $W_n$ are actually linear representations of $S_{t_n}$ (or $\mathcal A_{t_n}$).
\end{lemma}
\begin{proof}
Suppose the action of $S_{t_n}$ on $W_n$ is non-linear for almost all $n$. Then by the above theorem it follows that for almost all $n$ we have $\dim W_n \ge 2^{\lfloor \frac{t_n-4}{2}\rfloor}$ and hence $\dim W_n \ge 2^{t_n-5}$. So we get that for almost all $n$ $Mt_n^l \ge 2^{t_n-5}$, which is a contradiction since this inequality holds only for a finite number of $n$.

The same proof with $\dim W_n\ge 2^{t_n-6}$ instead of $\dim W_n \ge 2^{t_n-5}$ holds for $\mathcal A_{t_n}$.
\end{proof}

To prove the last lemma we will need to use another result, namely Lemma 2.8 from \cite{etingof2014representation}:
\begin{lemma}
For each $C > 0$ and $k \in \mathbb Z_+$ there exists $N(C,k) \in \mathbb Z_+$ such that for each $m>N(C,k)$, if $X(\mu)$ is an irreducible representation of $S_m$ which has dimension $\dim X(\mu)\le Cm^k$, then either the first row or the first column of $\mu$ has length $\ge m-k$.
\end{lemma}

Now to state our lemma we will need the following definitions:
\begin{def0}
 \textbf{a)} For an object of the symmetric rigid tensor category $W$  define $\mathfrak{gl}(W)$ to be an object $W \otimes W^*$. It has the structure of an associative algebra given by $1 \otimes ev \otimes 1 :\mathfrak{gl}(W)\otimes \mathfrak{gl}(W) \to \mathfrak{gl}(W)$, and thus has the structure of the Lie algebra.
 \textbf{b)} For an object of the symmetric rigid tensor category $W$  define $\mathfrak{sl}(W)$ to be a Lie algebra given by the kernel of the map $ev : \mathfrak{gl}(W) \to \bold 1$. In case of the category $Vect$ this algebra is simple iff the map $\bold 1\to \bold 1$ given by the composition of evaluation and coevaluation maps for $W$ is not zero.\\
  \textbf{c)} For an object of the symmetric rigid tensor category $W$  such that the above map $\bold 1\to \bold 1$ is zero, define $\mathfrak{psl}(W)$ to be the cokernel of the map $coev: \bold 1\to \mathfrak{sl}(W)$. In case of the category $Vect$ this algebra is simple.\\
  \textbf{d)}  For an object of the symmetric rigid tensor category $W$ equipped with a  {(skew-)} symmetric non-degenerate bilinear form (isomorphism $\psi: W \to W^*$), define $\mathfrak{so}(W)$($\mathfrak{sp}(W)$) to be the Lie subalgebra in $\mathfrak{gl}(W)$ given by the kernel of $\sigma \circ \psi \otimes \psi^{-1} + Id$. In the case of the category $Vect$ this algebra is simple.
\end{def0}

\begin{lemma}
Suppose $V$ is an object of $\Rep(S_t)$ given by the ultraproduct of $V_n \in \textbf{\Rep}_{p_n}(S_{t_n})$, almost all isomorphic to 
$\End(W_n)$ (or $\mathfrak{sl}(W_n)$, $\mathfrak{psl}(W_n)$, $\mathfrak{so}(W_n)$, $\mathfrak{sp}(W_n)$ if these objects are defined), for some $W_n \in \textbf{Rep}(S_{t_n})$. Then there exist $W'_n \in \textbf{Rep}(S_{t_n})$ such that $\End(W_n)\simeq \End(W'_n)$(or the corresponding Lie algebras are defined and isomorphic) and  $W=\prod_{\mathcal{F}}W'_n$ is an object of $\Rep(S_t)$. Hence $V = W\otimes W^*, \ \mathfrak{sl}(W), \ \mathfrak{psl}(W), \ \mathfrak{so}(W), \mathfrak{sp}(W)$.
\end{lemma}

\begin{proof}
Suppose $V = \sum_{i=1}^M \mathcal X(\lambda_i)$, where all $\mathcal X(\lambda_i)$ lie in  $[k]$ for $k \le L$. Then for almost all $n$ ($n>3$)
$$
\dim W_n \le \dim V_n \le M(t_n)^L \\ ,
$$
so by Lemma 2.2.4 there is an $N(M,L)$ such that for any $n>N(M,L)$ any irreducible representation $X(\mu)$ appearing in $W_n$ is such that either the first row or the first column of $\mu$ has length $t_n-L$.

Hence for almost all $n$ we have $W_n = \oplus X(\mu^{(n)}_j)\otimes \mathcal E^{(n)}_j$, with each $\mu^{(n)}_j$ having  first row of length at least $t_n-L$, and $\mathcal E^{(n)}_j$ being either one-dimensional trivial or sign representation. Let's denote by $W'_n$ the representation equal to $W_n' = \oplus X(\mu^{(n)}_j)$. 

We need to check several cases.

$\bullet$  $V_n =\End(W_n)$ \\ Since the number of summands in $V_n = W_n\otimes W_n$ is bounded and is bigger or equal than the number of summands in $W_n$ it follows that the number of summands in the latter representation is bounded. Hence, since there is a finite number of ways to put $L$ boxes into the rows of Young diagram, it follows that for almost all $n$ we have $W_n = \oplus X(\mu_j|_n)\otimes \mathcal E_j$, for some partitions $\mu_j$ of weight at most $L$. Now, if one of $\mathcal E_j$ is the sign representation and another $\mathcal E_i$ is the trivial representation then  the ultraproduct of $X(\mu_j|_n)\otimes X(\mu_i|_n) \otimes \mathcal E_j\otimes \mathcal E_i = (\oplus X(\eta_j|_n))\otimes sgn$ is not the object of $\Rep(S_t)$ since the number of rows is unbounded. But this contradicts $V$ being the object of $\Rep(S_t)$, hence  $\mathcal E_j$ are all trivial or all sign representations. In the latter case taking $W'_n$ instead of $W_n$ does not change $W_n\otimes W_n = W_n'\otimes W_n'$ (since $W'_n\otimes W'_n = W_n\otimes W_n \otimes sgn^{\otimes 2}$). But hence $W = \prod_{\mathcal F}W'_n = \oplus \mathcal X(\mu_j)$ is well defined.

$\bullet$ $\mathfrak{sl}$ or $\mathfrak{psl}$ \\ In first case we subtract one trivial representation of $S_{t_n}$ and in the second case two trivial representations, so it follows that the number of summands in $W_n$ is still bounded. Using the same reasoning as above it also follows that $\mathcal E_j$ are all trivial or are all sign representations since subtracting trivial representations from $W_n\otimes W_n$ cannot delete the representation with a big number of rows. So it follows that we again can take $W'_n$ and get the same $\End(W_n) = \End(W'_n)$ and hence the same Lie algebra. 

$\bullet$  $\mathfrak{so}$ \\ Let's write down $W_n$ as $W_n = \oplus X(\nu^{(n)}_j)\otimes U_j$. There all $\nu^{(n)}_j$ are different for different $j$ and $U_j$ are trivial representations of $S_{t_n}$. Now, an invariant symmetric bilinear form on $W_n$ is given by an isomorphism $\phi:W_n \to W_n^*$ . Since all irreducible representations of $S_n$ are real and self-dual, by Schur's lemma it follows that an isomorphism $\phi:W_n\to W_n^*$ decomposes to the sum of isomorphisms $\phi_j:X(\nu^{(n)}_j)\otimes U_j \to X^*(\nu^{(n)}_j)\otimes U^*_j$ given by tensor product of the isomorphism $X(\nu^{(n)}_j) \to X^*(\nu^{(n)}_j)$ and an isomorphism $\psi_j:U_j \to U_j^*$. So our symmetric invariant bilinear form is the sum of products of symmetric invariant forms on $X(\nu^{(n)}_j)$ and invariant forms on $U_j$. Thus forms on $U_j$ also should be symmetric. But then up to change of basis we can assume that the invariant bilinear form pairs $\oplus X(\mu^{(n)}_j)\otimes \mathcal E^{(n)}_j$ to itself in the decomposition  $W_n = \oplus X(\mu^{(n)}_j)\otimes \mathcal E^{(n)}_j$. Hence $\mathfrak{so}(W_n)$ contains copies of all tensor products of $\oplus X(\mu^{(n)}_j)\otimes \mathcal E^{(n)}_j \otimes \oplus X(\mu^{(n)}_i)\otimes \mathcal E^{(n)}_i$ for $i\ne j$. This means first that the number of summands in $W_n$ is bounded. And that the previous argument can be again used to prove that all $\mathcal E_j$ are either trivial or sign. So we can again take $W_n'$ instead of $W_n$ to obtain the same Lie algebra.

$\bullet$  $\mathfrak{sp}$ \\ Here the discussion in the previous paragraph can be repeated, but the invariant bilinear form on $U_j$ should be skew-symmetric. Hence all $U_j$ are even-dimensional. Again up to the change of basis in $U_j$, we can write down $W_n$ as $\oplus X(\mu^{(n)}_j)\otimes \mathcal E^{(n)}_j \otimes \overline{\mathbb F}_{p_n}^2$ where invariant form is given by sum of products of invariant forms on $\oplus X(\mu^{(n)}_j)\otimes \mathcal E^{(n)}_j$ with standard skew-symmetric form on $\overline{\mathbb F}_{p_n}^2$. Thus again it follows that $\mathfrak{sp}(W_n)$ contains copies of $\oplus X(\mu^{(n)}_j)\otimes \mathcal E^{(n)}_j \otimes \oplus X(\mu^{(n)}_i)\otimes \mathcal E^{(n)}_i$ for $i \ne j$. So all previous arguments can be repeated. 
\end{proof}

\section{Classification of simple associative algebras in $\Rep(S_t)$ and applications.}

\subsection{Classification.}

First we need to understand the classification of simple associative algebras for $\textbf{\Rep}_{p}(S_N)$. There is the following way of constructing such  algebras. Fix $G \subset S_N$ and a simple associative algebra $\Mat_m(\overline{\mathbb F}_p)$ with an action of $G$. From this information we can construct the algebra $\Fun_G(S_N,\Mat_m(\overline{\mathbb F}_p)) \in \textbf{\Rep}_p(S_N)$, which is equal to $\Ind_G^{S_N}(\Mat_m(\overline{\mathbb F}_p))$ as a representation. We have the following theorem (see for example \cite{etingof2017semisimple}, where it is formulated for any group):
\begin{thm}
Fix an algebraically closed field $k$. Any simple associative algebra in $\Rep(S_N,k)$ is isomorphic to \\ $\Fun_G(S_N,\Mat_m(k))$ and all such algebras are simple. Moreover $G$ is defined up to conjugation in $S_N$ and the action of $G$ on $\Mat_m(k)$ up to conjugation in $\Aut(\Mat_m(k))$.
\end{thm}

Now by \lthm simple associative algebras in $\Rep(S_t)$ are given by ultraproducts of simple associative algebras in $\textbf{\Rep}_{p_n}(S_{t_n}) =\mathcal C_n$ such that their ultraproduct as  objects of $\mathcal C_n$ lies in $\Rep(S_t)$.

So suppose $A \in \Rep(S_t)$ is a simple associative algebra in $\Rep(S_t)$, which is equal to the ultraproduct of $\Ind^{S_{t_n}}_{G_n}(B_n)$, where $B_n$ are matrix algebras. Then we can apply Lemma 2.2.1 to conclude that for almost all $n$ we have $G_n = S_{t_n-j}\times H$, $B = \prod_{\mathcal F}B_n$ is an object of $\Rep(S_{t-j})\boxtimes \Rep(H)$ and $A = \Ind_{S_{t-j}\times H}^{S_t}(B)$.

For the next step, we need to understand what sequences of $B_n$ are admissible and what we can obtain as a result of taking their ultraproduct. We know that $B_n=\Mat_{m_n}(\overline{\mathbb F}_{p_n})$ with a structure of representation of $S_{t_n-j}\times H$. Let's slightly change the notation and denote $B_n = \End(V_n)$, where $V_n$ are some finite-dimensional spaces over $\overline{\mathbb F}_{p_n}$. Since $S_{t_n-j} \times H$ acts by algebra automorphisms on $B_n$, we have a homomorphism $S_{t_n-j} \times H \to \Aut(B_n) = PGL(V_n)$. So we have a structure of projective representation of $S_{t_n-j}\times H$ on $V_n$. But note that $\dim V_n \le \dim B_n$ which is bounded by some $M(t_n-j)^L$ since $B =\prod_{\mathcal F}B_n$ is an object of Deligne category. So  Lemma 2.2.3 can be applied, and hence we conclude that the structure of representation of $S_{t_n-j}$ is linear and not projective.

Thus each $V_n$ is a representation of $S_{t_n-j}$ together with a projective action of $H$. Now we want to prove that $\prod_{\mathcal F}V_n$ as a representation of $S_{t_n-j}$ is a well-defined object of $\Rep(S_{t-j})$. But this follows from Lemma 2.2.6. So indeed we have $V = \prod_{\mathcal F}V_n$ an object of $\Rep(S_{t-j})$.

Now we are ready to state the classification theorem:
\begin{thm}\label{assoc}
Suppose $A$ is a simple associative algebra in $\Rep(S_t)$, then it is isomorphic to $Ind^{S_t}_{S_{t-j}\times H}(B)$, where $j \in \mathbb Z_+$, $H \subset S_j$ and $B$ equals to $V \otimes V^*$, where $V$ is an object of $\Rep(S_{t-j})$, together with an action of $H$ on $V \otimes V^*$ by algebra automorphisms. Any algebra obtained in this way is a simple associative algebra in $\Rep(S_t)$.

Moreover, $H$ is defined uniquely up to conjugation in $S_j$, and the structure of $H$-representation on $V\otimes V^*$ is defined uniquely up to conjugation inside $\Aut_{Ass-alg}(V\otimes V^*)$.
\end{thm}

\begin{proof}
From the above discussion it follows that $A = Ind^{S_t}_{S_{t-j}\times H}(B)$, for $B = \prod_{\mathcal F}B_n$ and  that $B_n = V_n \otimes V_n^*$ for $V_n$ with a projective action of $H$ and linear action of $S_{t_n-j}$ which commute with each other. Moreover $\prod_{\mathcal{F}}V_n=V$ is a well-defined object of $\Rep(S_{t-j})$. 

Suppose $V = \bigoplus \mathcal X(\lambda_j)\otimes \mathbb C^{k_j}$ with $\lambda_i \ne \lambda_j$ for $i\ne j$, then projective action of $H$ on each $V_n$ is given by the map $\rho_n:H \to \left (\bigoplus GL(\mathbb C^{k_j})\right)/(\{ c\cdot Id\})$. But there is a finite number of such  maps up to conjugacy, hence for almost all $n$ they are the same and we get the projective action of $H$ on $V$ itself. Which is the same as the action of $H$ on $V\otimes V^*$ by algebra automorphisms.

Also any such algebra is simple by \lthm.

Now we need only to check the uniqueness statement. Suppose we have two algebras $\Ind^{S_t}_{S_{t-j}\times H}(V\otimes V^*)$ and $\Ind^{S_t}_{S_{t-j'}\times H'}(W\otimes W^*)$. These algebras are isomorphic iff almost all algebras in the corresponding ultraproducts are isomorphic. But by Theorem 3.0.1 it follows that this is only possible iff $S_{t_n-j}\times H$ and $S_{t_n-j'} \times H'$ are conjugate in $S_{t_n}$ for almost all $n$, $\End(V_n) = \End(W_n)$ and the actions of $S_{t_n-j}\times H$ and $S_{t_n-j'} \times H'$ are conjugate in $\Aut(\End(V_n))$. So it follows that $j=j'$ (for $t_n > 2\max(j,j')$) for almost all $n$ and hence $H$ is conjugate to $H'$ inside $S_j$ (since the conjugation should leave $S_{t_n-j}$ invariant). Also it follows that $V_n$ and $W_n$ must have the same dimension, and since the action of $S_{t_n}$ on them is the same up to conjugation, we can assume that $V_n = W_n$ and they lead to the same object of $\Rep(S_t)$. Hence the last requirement is that the actions of $H$ and $H'$ on $\End(V_n)$ are conjugate. Hence by \lthm the statement of our Theorem follows.
 
\end{proof}

\begin{rem}
This gives us a classification of simple commutative algebras in $\Rep(S_t)$ given in \cite{sciarappa2015simple}, \cite{harman2016deligne} as a special case, where we restrict ourselves to $B$ being $1$-dimensional.
\end{rem}

\subsection{Application: functors between Deligne categories.}

In this section we will show how our result on classification of commutative algebras helps us classify symmetric tensor functors $\Rep (S_t) \to \Rep(S_{t'})$ and also their generalization $\Rep(S_t) \to \Rep(S_{t_1'}) \boxtimes \dots \boxtimes \Rep(S_{t_k'})$.

We will start with functors  $\Rep (S_t) \to \Rep(S_{t'})$, for $t,t' \notin \mathbb Z_{\ge 0}$. From Proposition 1.2.5 we know that all such functors are classified by commutative Frobenius algebras of dimension $t$ in $\Rep(S_{t'})$. We will start with the following lemma.

\begin{lemma}
Any commutative Frobenius algebra $A \in \Rep(S_t)$ is isomorphic to the direct sum of simple commutative algebras.
\end{lemma}
\begin{proof}
Using previous notation, we know from \lthm that $A$ corresponds to a sequence of objects $A_n \in \mathcal C_n$, with almost all of them being commutative Frobenius algebras. Now note that such an algebra cannot have a non-trivial radical. Indeed since $\sqrt{0}$ lies in every maximal ideal, it also lies in the kernel of $\Tr$. But then for any element $N \in \sqrt{0}$, we have $\Tr(a \cdot N) = 0$, hence the form is degenerate, which is a contradiction. So almost all $A_n$ are semisimple as commutative algebras in the category of vector spaces. Also all of them are finite-dimensional, since they are  objects of $\mathcal C_i$.

Thus it follows that the action of $S_{t_n}$ on such an $A_n$ arises from the action of $S_{t_n}$ on $\mspec (A_n)$. And now if the action of $S_{t_n}$ on $\mspec (A_n)$ has $l_n$ orbits with stabilizers $H_1,\dots, H_{l_n}$, it follows that $A_n = \bigoplus \Fun_{H_{j}}(S_{t_n}, \overline{\mathbb F_{p_n}})$. 

So almost all $A_n$ are semisimple commutative algebras as objects of $\mathcal C_i$, so by \lthm the same holds for $A$.
\end{proof}

We also need another lemma.

\begin{lemma}
Any semisimple commutative algebra $A \in \Rep(S_t)$ is a Frobenius algebra.
\end{lemma}
\begin{proof}
We know that $A = \bigoplus_{i=1}^N A_i$, where $A_i$ are simple algebras. Note that  $\Tr_{A}$ is equal to $(\Tr_{A_1}, \dots, \Tr_{A_N})$, so if we prove that each $A_i$ is Frobenius algebra it will follow that $A$ is too.

So consider $A=\Ind_{S_{t-j} \times H}^{S_t}(\bold 1)$. We want to prove that $\Tr \circ \mu$ is a non-degenerate form on $A$. To do that it is enough to prove that for $n$ such that $t_n > j$ the corresponding form on $A_n= \Ind_{S_{t_n-j}\times H}^{S_{t_n}}(\bold 1) = \Fun_{S_{t_n-j}\times H}(S_{t_n}, \bold 1)$ is non-degenerate. 

Consider functions $h_i$ which are zero everywhere except one conjugacy class of  \\ $S_{t_n-j}\times H$, where they are equal to $1$. Such functions give us a basis of $A_n$. Now obviously $h_ih_j = \delta_{ij}h_i$. Hence $\Tr(h_i) = 1$, and $\Tr(h_ih_j) = \delta_{ij}$. So the form is indeed non-degenerate, and we are done.
\end{proof}

From these two lemmas and Proposition 1.2.5 it follows:
\begin{prop}
All $\mathbb C$-linear symmetric tensor functors between Deligne categories $\Rep(S_t) \to \Rep(S_{t'})$ for $t,t'\notin \mathbb Z_{\ge 0}$ are in 1-1 correspondence with semisimple commutative algebras in $\Rep(S_{t'})$ of dimension $t$.
\end{prop}
\begin{proof}
From Proposition 1.2.5 we know that such functors are in 1-1 correpondence with commutative Frobenius algebras of dimension $t$, but from Lemmas 3.1.1 and 3.1.2 we know that any commutative Frobenius algebra is a semisimple commutative algebra and vice versa.
\end{proof}

Now since the dimension of $\Ind_{S_{t_n-k}\times H}^{S_{t_n}}(\mathbb C)$ is ${ t_n \choose k}  \frac{k!}{|H|}$, it follows that dimensions of simple commutative algebras in $\Rep(S_t)$ are multiples of ${ t \choose k}$ for some integer $k$, and thus all possible dimensions of commutative Frobenius algebras are positive integer linear combinations of ${t \choose k}$. Or in other words they can be described as $f(t)$, where $f \in R_+$ and $R_+$ is the subset in the algebra of integer valued polynomials, of positive integer linear combination of binomial coefficients. Hence we have the following Corollary.   
\begin{cor}
Symmetric monoidal functors between Deligne categories $\Rep(S_t) \to \Rep(S_{t'})$ with $t,t'\notin \mathbb Z_{\ge 0}$ exist iff $t = f(t')$ for some $f \in R_+$. 
\end{cor}

\begin{rem}
One can obtain a similar description for $\mathbb C$-linear symmetric tensor functors $\Rep(S_t) \to \Rep(S_{t_1'}) \boxtimes \dots \boxtimes \Rep(S_{t_k'})$. Such functors are in 1-1 correspondence with finite sums of external tensor products of simple commutative algebras in $\Rep(S_{t_i'})$. And such a functor exists iff $t$ is the positive integer linear combination of products of binomial coefficients in $t_1',\dots,t_n'$.
\end{rem}

To strengthen the result we can solve a simple number-theoretic problem.
\begin{lemma}
Consider a rational number $\frac{r}{s}$ with $r,s$ -- coprime positive integers and $s>1$. Let $p_1, \dots, p_m$ denote all prime numbers which divide $s$. Then the set $R_+(\frac{r}{s}) = \{ f(\frac{r}{s}) | \ f \in R_+ \}$ is equal to $\mathbb Z[1/s] = \{ \frac{n}{p_1^{k_1} \dots p_m^{k_m}} | k_i \in \mathbb Z_{\ge 0}\}$.
\end{lemma}
\begin{proof}
First note that $R_+(\frac{p}{q})$ is closed under multiplication. It is enough to prove that the product of ${ t \choose i}$ and ${ t \choose j}$ can be expressed as a positive integer linear combination of ${ t \choose k}$. This fact follows from the identity $(1+z)^t(1+w)^t = (1 + (z+w+zw))^t$. Indeed, expanding this we get:
$$
\sum_{i,j}{t \choose i} {t\choose j}z^iw^j = \sum_k {t \choose k} (z+w+zw)^k \ ,
$$
so it follows that ${t \choose i}{t \choose j}$ is equal to the coefficient of $z^iw^j$ in $\sum_k {t \choose k} (z+w+zw)^k$, which is going to be a positive integer linear combination of ${t \choose k}$ since all coefficients in the expansion of $(z+w+zw)^k$ are positive integers.

Now consider ${\frac{r}{s} \choose k} = \frac{r(r-s)(r-2s) \cdot \dots \cdot (r-(k-1)s)}{k!s^k}$. If we consider any prime number $q$ which does not divide $s$, then $s\cdot l$ is divisible by $p^d$ iff $l$ is divisible by $p^d$. Hence in the sequence $r, \ r-s, \ r-2s \ , \dots$ the numbers divisible by $p$ are $p$ terms apart, divisible by $p^2$ -- $p^2$ terms apart and so on. Hence the valuation $\nu_{p}(\prod_{j=0}^k (r-js)) \ge \nu_{p}(\prod_{j=0}^k(1+j))$.

Thus if we write $k! = h(k)g(k)$, where $h(k)$ is the part which contains the product of all prime numbers which do not divide $s$, and $g(k)$ is the remainder, then we have: $\frac{r(r-s)\dots(r-(k-1)s)}{k!} = \frac{N}{g(k)}$ for some integer $N$. So it follows that:
$$
{\frac{r}{s} \choose k}  = \frac{N}{g(k)s^k} \in  \mathbb Z[1/s] .
$$
Since $\mathbb Z[1/s]$ is obviously closed under addition, it follows that $R_+(\frac{r}{s}) \subset \mathbb Z[1/s]$, so we need to show the other inclusion.
Consider first $k$ such that $r < s(k-1)$, then ${\frac{r}{s} \choose k}$ turns out to be negative. So we have an element $\frac{N}{M}$ with $N < 0$ and $N, M$ -- coprime in $R_+(\frac{r}{s})$. We also have $1 \in R_+(\frac{r}{s})$. Since $N$ and $M$ are coprime it follows that there are integers $a,b$ such that $aN +bM = 1$. Note that if $a<0$ and $b>0$, then the sum is at least $N+M>1$. The same problem arises with $a>0$ and $b<0$. So $a,b$ have the same sign. So making them positive we can arrange for $aN+bM = \pm 1$. So it follows that $a\frac{N}{M} + b = \frac{\pm 1}{M} \in R_+$ and since $R_+$ is closed under multiplication $\frac{1}{M^2} \in R_+(\frac{r}{s})$. Note that since $M$ was divisible at least by $s$, it is divisible by all primes $p_1, \dots, p_m$. Hence by multiplying by an integer we conclude that $\frac{1}{p_i} \in R_+(\frac{r}{s})$, hence by taking products an multiplying by positive integers all $\frac{n}{p_1^{k_1} \dots p_m^{k_m}}$ with $n\ge 0$ are in $R_+(\frac{r}{s})$.

So, since $-N-1 \ge 0$, it follows that $\frac{-N-1}{M}$ is in $R_+(\frac{r}{s})$, thus $\frac{-1}{M}$ is in $R_+(\frac{r}{s})$ and by the same logic all $\frac{-1}{p_i} \in R_+(\frac{p}{q})$. Thus $\mathbb Z[1/s] \subset R_+(\frac{r}{s})$ and we conclude that $R_+(\frac{r}{s})=\mathbb Z[1/s]$.
\end{proof}

\begin{rem}
Note that since $r,s$ are coprime, it follows that $\mathbb Z[\frac{1}{s}] = \mathbb Z[\frac{r}{s}]$. Also it is easy to see that if $l \in \mathbb Z_{<0}$, then $\mathbb Z[l] = \mathbb Z$ and $R_+(l) = \mathbb Z$. 
\end{rem}

Now we can apply this lemma to get an interesting result about functors into Deligne categories.

\begin{cor}
A symmetric monoidal functor between Deligne categories $\Rep(S_t) \to \Rep(S_{t'})$ with $t' \in \mathbb Q \backslash \mathbb Z_{\ge 0}$ exists iff $t\in \mathbb Z[t']$. 
\end{cor}

\section{Classification of simple Lie algebras in $\Rep(S_t)$.}

First we need to state the classification theorem for Lie algebras in characteristic $p$. See chapter 4 of \cite{strade2004simple}.

\begin{thm}
Suppose $\mathfrak g$ is a simple finite-dimensional Lie algebra over an algebraically closed field of characteristic $p>5$. Then it is either of classical or Cartan type.
\end{thm}

Now we need to explain that classical and Cartan type mean. First, classical type Lie algebras can be obtained in the following way. Take any Dynkin diagram $C$ and define the Lie algebra $\mathfrak g_C$ as the vector space spanned by Chevalley basis corresponding to $C$ with the ordinary Chevalley relations taken modulo $p$. It turns out that this algebra is simple for any $C$ except $A_{kp-1}$ for a positive integer $k$. In this case we also need to take quotient by $1$-dimensional center of $\mathfrak{sl}_{kp}$ spanned by scalar matrices and we get the simple algebra $\mathfrak{psl}_{kp}$.

Algebras of Cartan type form four series of simple Lie algebras, namely $W(m, \underline{n})$, $S(m, \underline{n})$, $H(m,\underline{n})$ and $K(m,\underline{n})$, where $m \in \mathbb Z_{> 0}$ and $\underline{n} \in \mathbb Z_{> 0}^{m}$ (in the last case $m$ is odd, in the second to last case it is even). We will discuss some of their properties in the next subsection.

The result analogous to Theorem 3.0.1 also holds in the case of Lie algebras, we only need to exchange word ``associative" to ``Lie" in the statement of the Theorem (\cite{etingof2017semisimple}). 
\begin{thm}
Fix an algebraicly closed field $k$.
Any simple Lie algebra in $\Rep(S_N,k)$ is isomorphic to $\Fun_G(S_N,\mathfrak h)$, for a Lie algebra $\mathfrak h$ simple in a category of vector spaces. All such algebras are simple. Moreover $G$ is defined up to conjugation in $S_N$ and the action of $G$ on $\mathfrak h$ up to conjugation in $\Aut(\Mat_m(k))$.
\end{thm}

We can now state the following Proposition:
\begin{prop}
Any simple Lie algebra $\mathfrak{g} \in \Rep(S_t)$ is equal to $Ind_{S_{t-j} \times H_j}^{S_{t}} \mathfrak{h}$, for some $j$ and $H \subset S_j$, where $\mathfrak{h}$ is a simple Lie algebra given by the ultraproduct of simple Lie algebras $\mathfrak{h}_n \in \mathcal C_n$ which remain simple under the forgetful functor $Res: \mathcal C_n \to Vec$.
\end{prop}
\begin{proof}
Since we know that $\mathfrak{g} = \prod_{\mathcal F}\mathfrak{g}_n$ and $\mathfrak{g}_n = \Ind^{S_{t_n}}_{G_n}(\mathfrak{h}_n)$, where $\mathfrak{h}_n$ is simple lie algebra as an object of the category of vector spaces, the result follows from Lemma 2.2.1 and \lthm.
\end{proof}

\subsection{Ultraproducts of Lie algebras of Cartan type}

We want to  rule out the case of almost all $\mathfrak{h}_n$ being of Cartan type.

To do this let's first explain what $W(m, \underline{n})$ actually is (see chapter 4.2 in \cite{strade2004simple} for details). First, we need to define $\mathcal O(m)$ and $\mathcal O(m,\underline{n})$.

\begin{def0}
By $\mathcal O (m)$ denote the commutative algebra over $\overline{\mathbb {F}}_q$ with basis $x_1^{(a_1)}\dots x_m^{(a_m)}$, for $a_i \in \mathbb Z_{\ge 0}$ with multiplication defined by:
$$
x_1^{(a_1)}\dots x_m^{(a_m)}\cdot x_1^{(b_1)}\dots x_m^{(b_m)} = {a_1+b_1 \choose a_1} \dots {a_m+b_m \choose a_m}x_1^{(a_1+b_1)}\dots x_m^{(a_m+b_m)}\ .
$$
By $\mathcal O(m,\underline{n})$ denote the subalgebra of $\mathcal O(m)$ spanned by $x_1^{(a_1)}\dots x_m^{(a_m)}$ with $0 \le a_i < p^{n_i}$.
\end{def0}

Using this, the Witt algebra $W(m,\underline{n})$ can be obtained in the following way.

\begin{def0}
 By $W(m,\underline{n})$ denote the simple Lie algebra given as follows:
 $$
 W(m,\underline{n}) = \sum_{i=1}^m \mathcal O(m,\underline{n})\partial_i \ .
 $$
\end{def0}

All other simple algebras $S(m, \underline{n})$, $H(m,\underline{n})$ and $K(m,\underline{n})$ can be realized as a subalgebras in $W(m,\underline{n})$. We will need the important proposition about the automorphism groups of such algebras, see chapter 7.3 of \cite{strade2004simple}.

\begin{prop}
There is an isomorphism $\phi: \Aut_{\mathcal C}(\mathcal O(m,\underline{n})) \to \Aut(W(m,\underline{n}))$ from a certain subgroup $\Aut_{\mathcal C}(\mathcal O(m,\underline{n})) \subset \Aut(\mathcal O(m,\underline{n}))$, to the group of automorphisms of Witt algebra given by:
$$
\sigma \to (D \mapsto \sigma \circ D \circ \sigma^{-1}) \ .
$$
Moreover it restricts to give an isomorphism between $S(m,\underline{n})$, $H(m,\underline{n})$ and $K(m,\underline{n})$ and certain subgroups of $\Aut_{\mathcal C}(\mathcal O(m,\underline{n})$.
\end{prop}

So from this proposition it follows what we need to understand the structure of the group $\Aut_{\mathcal C}(\mathcal O(m,\underline{n}))$. This is done in \cite{wilson1971classification}. See Corollary 1 and Theorem 2.
\begin{prop}
Take any isomorphism $\sigma \in \Aut_{\mathcal C}(\mathcal O(m,\underline{n}))$, denote by $y_i$ the images of $x_i = x_i^{(1)}$ under this morphism. By $\bar{y}_i$ denote the linear part of $y_i$. It follows that the map $x_i \mapsto \bar{y}_i$ defines an element of $GL(m,\overline{\mathbb F}_p)$.
Also there is an exact sequence:
$$
0 \to \mathcal B \to \Aut_{\mathcal C}(\mathcal O(m,\underline{n})) \to GL(m,\overline{\mathbb F}_p) \ .
$$
Where $\mathcal B$ is solvable and the last morphism is as described above.
\end{prop}

Since the automorphism groups of all Cartan type Lie algebras are soubgroups in $\Aut_{\mathcal C}(\mathcal O(m,\underline{n}))$\footnote{Theorem 7.3.2 \cite{strade2004simple}.} it follows that such an exact sequence holds for any $\Aut(X(m,\underline{n}))$ with $X = W,S,H,K$. Namely we have:
\begin{equation}
    0 \to \mathcal{B_X} \to \Aut(X(m,\underline{n})) \to  GL(m,\overline{\mathbb F}_p) \ ,
\end{equation}
for different $\mathcal{B_X}$.

Also we will need to know the dimension formulas for Cartan type Lie algebras, they are summarized in the folowing proposition (see \cite{strade2004simple} section 4.2).
\begin{prop}
The dimension of Cartan type algebras are given by the formulas $\dim(W(m,\underline{n})) =mp^{\sum n_i}$, $\dim(S(m,\underline{n})) = (m-1)(p^{\sum n_i}-1)$, $\dim(H(m,\underline{n})) = p^{\sum n_i}-2$ and $\dim(K(m,\underline{n}))= p^{\sum n_i}$ or $p^{\sum n_i}-1$ depending on $m \ mod \ p$.
\end{prop}

Now we have everything we need to move on. So let's prove the following proposition.

\begin{prop}\label{nocartan}
In Proposition 4.0.3 almost all $\mathfrak{h}_n$ are of classical type.
\end{prop}
\begin{proof}
Suppose that almost all $\mathfrak h_n$ in Proposition 4.0.3 are of Cartan type. Let's denote $\mathfrak{h}_n = X_n(m_n, \underline{N_n})$ ($X_n = W,S,H,K$), then we have a homomorphism $S_{t_n-j} \to \Aut(\mathfrak{h}_n)$. Hence, because of (2), we have $S_{t_n-j} \to GL(m,\overline{\mathbb F}_{p_n})$. There are two possibilities here. Either for almost all $n$ this homomorphism is trivial or not. 

First suppose it is trivial for almost all $n$. When for almost all $n$, $S_{t_n-j} \to \mathcal B_n$, but since the latter group is solvable, it follows that for almost all $n$ the kernel of this morphism contains $\mathcal A_n$. But then $\mathfrak{h}_n$ contains only one-dimensional representations of $S_{t_n-j}$. But since the dimension of $\mathfrak{h}_n$ is bigger than $p_n-3$ it follows that the length of $\mathfrak{h}_n$ as a representation of $S_{t_n-j}$ is unbounded, hence its ultraproduct does not define an object of the Deligne category.

The only other option is that this morphism is non-trivial for almost all $n$. Note that this morphism cannot have $\mathcal{A}_n$ as its kernel or the previous argument can be repeated. Hence, since the lowest dimension of $S_{t_n-j}$-representation which is not trivial or sign  is $t_n-j-1$, it follows that $m_n \ge t_n-j-1$, and thus the dimension of $\mathfrak{h}_n$ is bigger or equal then $p_n^{t_n-j-1}-3$. So it grows exponentially. But the dimension of any sequence of representations defining an element of the Deligne category grows polynomialy. So again ultraproduct of $\mathfrak{h}_n$ does not belong to $\Rep(S_{t-j})$.

Thus the  result follows.
\end{proof}

\subsection{Ultraproducts of classical Lie algebras}

So, now we can assume that for almost all $\mathfrak{h}_n$ are of classical type. 

Here we again have two possibilities. Either the size of the Dynkin diagram coresponding to $\mathfrak{h}_n$ is bounded for almost all $n$ or it is not.
Let's start with  the first case.
\begin{prop}
If the size of Dynkin diagram corresponding to  $\mathfrak{h}_n$ is bounded, then $\mathfrak{h}$ has trivial action of $S_{t-j}$.
\end{prop}
\begin{proof}
 Since there are finite number of Dynkin diagrams of bounded size, it follows that for almost all $n$ the corresponding Dynkin diagram is the same, so as a Lie algebra in the category of vector spaces, $\mathfrak{h}_n$ is of the same type. But then its automorphism group is a subgroup in some $GL_N(\overline{\mathbb F}_{p_n})$. So since the lowest dimension of irreducible $S_{t_n-j}$-representation which is not trivial or sign is $t_n-j-1$ it follows that almost all $\mathfrak{h}_n$ are sums of one-dimensional representations of $S_{t_n-j}$. But it cannot contain the sign representations  for almost all $n$, or the ultraproduct wouldn't lie in $\Rep(S_{t-j})$. Hence for almost all $n$, $\mathfrak{h}_n$ is a trivial representation of $S_{t_n-j}$. Thus the coresponding ultra-product is the same classical Lie algebra corresponding to the Dynkin diagram with a trivial action of $S_t$, i.e. equal to the sum of unit objects in $\Rep(S_{t-j})$.
\end{proof}

In the second case the size of the Dynkin diagram is unbounded. But the number of infinite series of Dynkin diagrams is finite, so we may assume that for almost all $n$ the type of the Dynkin diagram is the same, and it is either $A,B,C$ or $D$. To proceed further we need to know something about the automorphism groups of these algebras. This information can be found in \cite{seligman1960automorphisms}, it is summarized in the following Proposition.

\begin{prop}
The group of automorphisms of the Lie algebra of type $A_{n-1}$ (both in the case $p | n$ and $p \nmid n$) is the semi-direct product of $PSL(n)$ by $\mathbb Z/ 2\mathbb Z$, where the second group acts by $X \mapsto -X^t$. We will denote the generator of this group by $\tau$.

The group of automorphisms of the Lie algebra of type $B_{n}$ is $PSp(n)$.

The group of automorphisms of the  Lie algebra of type $C_{n}$ is $PSO(2n+1)$.

The group of automorphisms of the Lie algebra of type $D_{n}$ is $PO(2n)$, for $n>4$.
\end{prop}

We have an additinal complication in the $A_{n-1}$ case, so let's first sort this out.

\begin{prop}
If $\mathfrak{h}_n$ is the simple Lie algebra of type $A$ for almost all $n$, then, for almost all $n$, $S_{t_n-j}$ maps into the subgroup $PSL(n)$ of automorphisms of $\mathfrak{h}_n$. 
\end{prop}

\begin{proof}
Suppose the map $S_{t_n-j} \to \mathbb Z/2\mathbb Z$ obtained as a composition of maps $S_{t_n-j} \to \Aut(\mathfrak{h}_n)$ and $\Aut(\mathfrak{h}_n) \to \mathbb Z / 2 \mathbb Z$ is non-trivial for almost all $n$. Note that the map $\mathcal A_{t_n-j}$ obtained by restriction of this map is trivial, so $\mathcal A_{t_n-j}$ maps into $PSL(N_n)$. Suppose this map is non-trivial. By Lemma 2.2.3 it follows that this map gives us a linear representation of $\mathcal A_{t_n-j}$ on $V = \overline{\mathbb F_{p_n}}^{N_n}$. By choosing a bilinear $\mathcal A_{t_n-j}$-invariant form on $V$, we can suppose that $V^* \simeq V$ as a representation of $\mathcal A_{t_n-j}$. But then since $\tau$ acts on an element of $V\otimes V^*$ as  $\tau(v\otimes w) = -w \otimes v$, it follows that $\tau$ commutes with the action of $\mathcal A_{t_n-j}$. So the action of $\tau$ on the isomorphic image of $\mathcal A_{t_n-j}$ is actually trivial, and hence the image of $S_{t_n-j}$ is actually a direct product of $\mathcal A_{t_n-j}$ and $\mathbb Z/2\mathbb Z$, which is absurd. So $\mathcal A_{t_n-j}$ is  in the kernel of the map $S_{t_n-j} \to \Aut(\mathfrak{h}_n)$. So for almost all $n$, $\mathfrak{h}_n$ decomposes as the sum of one-dimensional representations of $S_{t_n-j}$, hence the action of $S_{t_n-j}$ on $\mathfrak h_n$ is actually trivial since the ultraproduct lies in the Deligne category. So we get a contradiction.
\end{proof}

From this proposition it follows that in each case $S_{t_n-j}$ maps into a projective group of the corresponding group of linear transformations of a vector space. But from Lemma 2.2.3  it follows that in each case we have an honest map from $S_{t_n-j}$ to the corresponding group of linear transformations, i.e. a represnetation of $S_{t_n-j}$ on the corresponding vector space.

We have the following proposition:
\begin{prop}
If the size of Dynkin diagram corresponding to  $\mathfrak{h}_n$ is unbounded, then for almost all $n$, $\mathfrak{h}_n = \mathfrak{x}(V_n)$ for the same $\mathfrak x$ ($\mathfrak{x} = \mathfrak{sl}, \mathfrak{psl},\mathfrak{sp}, \mathfrak{so}$). Also there exist $V'_n$ such that $\mathfrak{x}(V_n) =\mathfrak{x}(V'_n)$ and such that $V = \prod_{\mathcal F}V'_n$ is an object of $\Rep_{S_{t-j}}$, hence $\mathfrak h = \mathfrak x(V)$.
\end{prop}
\begin{proof}
We have established the first part of the proposition. Also from the discussion above it follows that the action of $S_{t_n-j}$ the representation of $S_{t_n-j}$ on $V_n,$ which leaves the correponding bilinear form invariant. 
Now using Lemma 2.2.6 we conclude that such $V_n'$ indeed exist.
\end{proof}

Now we can formulate the following classification theorem.
\begin{thm} \label{lieclass}
To construct a simple Lie algebra in the category $\Rep(S_t)$ one needs to fix an integer $j$, a subgroup $H \subset S_j$ and the Lie algebra $\mathfrak h$ in $\Rep(S_{t-j})$ of one of the following kinds:

$\cdot$ An exceptional Lie algebra which is equal to the sum of unit objects of $\Rep(S_{t-j})$.

$\cdot$ $\mathfrak{sl}(V)$ for any $V$ of dimension not zero, or $\mathfrak{psl}(V)$ for any $V$ of dimension zero.

$\cdot$ $\mathfrak{so}(V)$ or $\mathfrak{sp}(V)$ for any $V$ with a (skew)-symmetric non-degenerate bilinear form.

Also one should fix an action of $H$ on $\mathfrak{h}$ by  Lie algebra automorphisms. Then one can obtain a simple Lie algebra $\mathfrak{g}$ in $\Rep(S_t)$ as $\Ind_{S_{t-j}\times H}^{S_t}(\mathfrak h)$.

Moreover any simple Lie algebra in $\Rep(S_t)$ is isomorphic to one obtained in this way. 

Finally, such a simple Lie algebra is determined uniquely by the above data up to conjugation of $H$ inside $S_j$ and conjugation of action of $H$ inside of $\Aut(\mathfrak h)$.
\end{thm}

\begin{proof}
That the above process gives us a simple Lie algebra is straightforward, since the resulting algebra $\mathfrak{g}$ is an ultraproduct of Lie algebras which are simple due to Theorem 4.0.2.

Now from Propositions 4.0.3, 4.1.6, 4.2.1 and 4.2.4 we conclude that any simple Lie algebra can be obtained in this way. Indeed from these propositions we know that such $\mathfrak h$ exists and is either a trivial representation of $S_{t-j}$ or it is given by $\mathfrak x(V)$. Now note that if $\mathfrak h =\mathfrak{sl}(V)$, the dimension of almost all $V_n$ are not divisible by $p_n$, or the algebra $\mathfrak{sl}(V_n)$ would not be simple for almost all $n$, hence the dimension of $V$, which can be obtained through the isomorphism $\prod_{\mathcal{F}}\overline{\mathbb F}_{p_n} = \mathbb C$ is non-zero. In the case of $\mathfrak{psl}(V)$ on the other hand it is divisible by $p_n$ for almost all $n$ and hence the dimension of $V$ is zero. In the case $\mathfrak x = \mathfrak{so}, \ \mathfrak{sp}$ the $S_{t_n-j}$-module $V_n$ has an $S_{t_n-j}$-invariant (skew)-symmetric bilinear form for almost all $n$, hence it gives us the (skew)-symmetric bilinear form on $V$ in $\Rep(S_t)$. So indeed every simple Lie algebra can be obtained in the specified way.

The proof of the uniqueness is the same as in Theorem 3.0.2.

\end{proof}

\begin{rem}
This theorem can likely be extended to the degenerate case when $t \in \mathbb Z_{\ge 0}$. $\Rep(S_t)$ is not abelian for $t \in \mathbb Z$ so instead one should work with the abelian envelope $\Rep^{ab}(S_t)$ defined in \cite{comes2014deligne}. This abelian envelope has a lot of nice properties, for example it is a highest-weight category (see \cite{barter2017deligne}), and it still can be similarly interpreted via an ultrafilter construction as outlined in \cite{harman2016deligne}. The main difference is that the categories $\Rep^{ab}(S_t)$ are not semisimple so some care needs to be taken in some of the arguments. However the main technical step of ruling out the Cartan type Lie algebras (Proposition \ref{nocartan}) by looking at their dimension growth goes through as is.
\end{rem}

\section{Conjecture concerning classification of simple Lie superalgebras in $\Rep(S_t)$}

We will state a conjectural extension of the main results of this paper to the setting of Lie superalgebras, and outline a possible approach to generalize the methods in this paper. The textbook reference about the theory of Lie superalgebras is \cite{musson2012lie}, it contains the classification of simple Lie superalgebras over $\mathbb C$ and their construction (Chapters 1-2 and 4). See the original paper of Kac $\cite{kac1977lie}$ for the classification.\footnote{The classification problem of finite dimensional simple Lie superalgebras in zero and positive characteristic has a long history, which we in no way attempt to review; thus many important references are not given here.} How these results generalize to the modular case with big enough $p$ can be found in \cite{bouarroudj2015simple} Section 2.3, \cite{leites2007towards} and \cite{leites2009classification} Section 10.

\subsection{Superalgebras in tensor categories and their simplicity in $Vect$}

Below we assume that for every Lie superalgebra $\mathfrak{g} = \mathfrak{g}_0 \oplus \mathfrak{g}_1$, the $\mathfrak{g}_1$ component is non-zero. If $\mathfrak{g}_1$ is zero, then  $\mathfrak{g}$ is a regular Lie algebra and the result of the previous section applies.

First we will need some definitions.

\begin{def0}
 Fix $V,W$ to be non-zero objects of a symmetric rigid tensor category. \\
 \textbf{a)} Define the Lie superalgebra $\mathfrak{gl}(V|W)$ to be the object $(V \oplus W) \otimes (V^* \oplus W^*)$ with the $\mathbb Z/2\mathbb Z$-grading given by $\mathfrak{gl}(V|W)_0 = V \otimes V^* \oplus W \otimes W^*$ and $\mathfrak{gl}(V|W)_1 = V \otimes W^* \oplus W \otimes V^*$. The superbracket $[\ ,\ ]_{i,j}: \mathfrak{gl}(V|W)_i \otimes \mathfrak{gl}(V|W)_j \to \mathfrak{gl}(V|W)_{i+j}$ is given by $\mu -(-1)^{ij} \mu \circ \sigma$, where $\mu$ is the associative algebra multiplication and $\sigma$ is an operator permuting copies of $\mathfrak{gl}(V|W)$. \\
 \textbf{b)} Define the Lie superalgebra $\mathfrak{sl}(V|W)$ to be the subalgebra in $\mathfrak{gl}(V|W)$ given by the kernel of the map $str: \mathfrak{gl}(V|W) \to  \mathfrak{gl}(V|W)_0 \xrightarrow{(ev_V \ , \ -ev_W)} \bold 1$.  \\
 \textbf{c)} Consider the map $l:\bold 1\xrightarrow{coev_V \oplus coev_W} \mathfrak{gl}(V|W)_0$. The image of this map lies in $\mathfrak{sl}(V|W)$ iff $\dim V = \dim W$. In this case define the Lie superalgebra $\mathfrak{psl}(V|W)$ to be the cokernel of $l: \bold 1\to \mathfrak{sl}(V|W)$. \\
 \textbf{d)} Fix a bilinear form on $V \oplus V^*$ specified by the identity map $\psi:V \oplus V^*\to (V\oplus V^*)^* = V^* \oplus V$. Define the Lie superalgebra $\mathfrak{p}(V)$ to be the subalgebra in $\mathfrak{sl}(V|V^*)$ preserving this form, i.e. the kernel of the map $\mathfrak{gl}(V|V^*) \xrightarrow{1 + \sigma \circ (\psi \otimes \psi^{-1})} \mathfrak{gl}(V|V^*)$.  \\
 \textbf{e)} Consider a morphism $c: V \oplus V \to V \oplus V$, given by the matrix $\left( \begin{matrix} 0 & Id \\ -Id & 0 \end{matrix}\right)$. This morphsim can  also be considered as an element (i.e. a map $\bold 1\to \mathfrak{gl}(V|V)$) of $\mathfrak{gl}(V|V)$ using evaluation and coevaluation maps. The Lie superalgebra $\hat{\mathfrak{q}}(V)$ is defined as the centralizer of this  element, i.e. as the kernel of the map $\mathfrak{gl}(V|V) \to \mathfrak{gl}(V|V) \otimes \bold 1\xrightarrow{Id \otimes c} \mathfrak{gl}(V|V) \otimes \mathfrak{gl}(V|V) \xrightarrow{[\ ,\ ]} \mathfrak{gl}(V|V)$. Next we define the Lie superalgebra $\tilde{\mathfrak{q}}(V)$ as the kernel of the map $\hat{\mathfrak{q}}(V) \to \bold 1$ given by the restriction of the map $\mathfrak{gl}(V|V)_1 \xrightarrow{(ev_V \ , \ ev_V)} \bold 1$. Then there is a non-zero map $l:\bold 1\xrightarrow{coev \oplus coev} \tilde{\mathfrak{q}}(V)_0$. The cokernel of this map is a Lie superalgebra $\mathfrak{q}(V)$. \\
 \textbf{f)} Suppose there is a symmetric non-degenerate bilinear form on $V$ and a skew-symmetric non-degenerate bilinear form on $W$. Denote the corresponding maps $\psi_V:V \to V^*$ and $\psi_W: W \to W^*$. Then $\mathfrak{osp}(V|W)$ is the following subalgebra in $\mathfrak{gl}(V|W)$.
 We define $\mathfrak{osp}(V|W)_0$ to be equal to the kernel of the map 
 $$
 \mathfrak{gl}(V|W)_0 \xrightarrow{(Id +\sigma \circ (\psi_V \otimes \psi_V^{-1}))\oplus (Id +\sigma \circ (\psi_W \otimes \psi_W^{-1}))} \mathfrak{gl}(V|W)_0 \ ,
 $$
 and $\mathfrak{osp}(V|W)_1$ to be equal to the kernel of the map
 $$
 \mathfrak{gl}(V|W)_1 = V\otimes W^* \oplus W\otimes V^* \xrightarrow{(\sigma \circ (\psi_V\otimes 1) \ , \ 1 \otimes \psi_W)} W^* \otimes V^* \ .
 $$
\end{def0}

\begin{rem}
These definitions mimic the standard definitions of the above Lie superalgebras in an element-free fashion. It is straightforward to check that this definition agrees with the usual definitions for the category of vector spaces, and that the superbracket descends onto the various kernels and cokernels used in the definition. 
\end{rem}

Now we want to know when exactly these superalgebras are simple for categories $Vect_{k_0}$ and $Vect_{k_p}$. This is explained by the various classification results. Here $k_0$ stands for an algebraically closed field of characteristic $0$ and $k_p$ stands for an algebraically closed field of characteristic $p$.

\begin{prop} (Theorem 1.3.1 in \cite{musson2012lie})
Suppose $V,W$ are non-zero objects of $Vect_{k_0}$.

\textbf{a)} The Lie superalgebra $\mathfrak{sl}(V|W)$ is simple in $Vect_{k_0}$ iff  $\dim(V)\ne \dim(W)$. \\
\textbf{b)} In $Vect_{k_0}$ if $\dim(V) = \dim(W) > 1$ the Lie superalgebra $\mathfrak{psl}(V|W)$ is defined and is simple.  \\
\textbf{c)} The Lie superalgebra $\mathfrak{osp}(V|W)$ is simple in $Vect_{k_0}$. \\
\textbf{d)} The Lie superalgebra $\mathfrak{q}(V)$ is simple in $Vect_{k_0}$ iff $\dim(V) \ge 2$. \\
\textbf{e)} The Lie superalgebra $\mathfrak{p}(V)$ is simple in $Vect_{k_0}$ iff $\dim(V) \ge 2$. 
\end{prop}

\begin{prop} (Section 10 in \cite{leites2009classification}, Section 6 in \cite{bouarroudj2008deforms}, Section 4.1 in \cite{bouarroudj2013derivations})
Suppose $V,W$ are non-zero objects of $Vect_{k_p}$.

\textbf{a)} The Lie superalgebra $\mathfrak{sl}(V|W)$ is  simple in $Vect_{k_p}$ iff $\dim(V) \ne \dim(W) \  mod \ p$ (See Section 10 in \cite{leites2009classification}).\\
\textbf{b)} In $Vect_{k_p}$ if $\dim(V) = \dim(W) \ mod \ p$ and $\dim(V),\dim(W)> 1$, the Lie superalgebra $\mathfrak{psl}(V|W)$ is defined and is simple (See Section 10 in \cite{leites2009classification}). \\
\textbf{c)} The Lie superalgebra $\mathfrak{osp}(V|W)$ is simple in $Vect_{k_p}$. (See Section 10 in \cite{leites2009classification}). \\
\textbf{d)} The Lie superalgebra $\mathfrak{q}(V)$ is simple in $Vect_{k_p}$ iff $\dim(V) \ge 2$. 
\\
\textbf{e)} The Lie superalgebra $\mathfrak{p}(V)$ is simple in $Vect_{k_p}$ iff $\dim(V) \ge 2$.
\end{prop}

We have some results about the classification of all such superalgebras.
\begin{thm} (Theorem 1.3.1 in \cite{musson2012lie}, Section 4.2 in \cite{kac1977lie})
Let $\mathfrak{g}$ be a finite dimensional simple Lie superalgebra over $k_0$. Then it is either given by one of the examples of Proposition 5.1.2 or by one of the exceptional Lie superalgebras $\mathfrak{d}(2,1; \alpha)$, $\mathfrak{f}(4)$ or $\mathfrak{g}(3)$ or by one of the Cartan type superalgebras $W(n)$ , $S(n)$, $\tilde{S}(n)$ and $H(n)$.
\end{thm}

\begin{rem}
The Lie superalgebras $\mathfrak{d}(2,1;\alpha)$ form a one-parametric series of superalgebras of the same dimension.
\end{rem}

\begin{conj} (Conjecture 1 in \cite{leites2007towards})
Let $\mathfrak{g}$ be a finite dimensional simple Lie superalgebra over $k_p$ with $p \ge 7$. Then it is either given by  one of the examples of Proposition 5.1.3 or by one of the exceptional Lie superalgebras or by a certain algebra of Cartan type.
\end{conj}

\subsection{Lie superalgebras in $\Rep(S_t)$}

Using Definition 5.1.1 we can construct the Lie superalgebras in  the category $\Rep(S_t)$ as follows. Fix an integer $j$, a subgroup $H \subset S_j$ and a Lie superalgebra $\mathfrak h$ in $\Rep(S_{t-j})$ of one of the following kinds:

$\cdot$ An exceptional or a Cartan type Lie superalgebra which as an object of $\Rep(S_{t-j})$ is equal to the sum of unit objects.

$\cdot$ $\mathfrak{sl}(V|W)$ for $V,W$ such that $\dim(V) \ne \dim(W)$ and $V,W \ne 0$.

$\cdot$ $\mathfrak{psl}(V|W)$ for $V,W$ such that $\dim(V) = \dim(W)$ and both objects are not $0$ or $\bold 1$.

$\cdot$ $\mathfrak{osp}(V|W)$ for any pair of non-zero objects $V,W$ with non-degenerate billinear form, which is symmetric and skew-symmetric respectively.

$\cdot$ $\mathfrak{q}(V)$ for $V$ not $0$ or $\bold 1$.


$\cdot$ $\mathfrak{p}(V)$ for $V$ not $0$ or $\bold 1$.


Also fix an action of $H$ on $\mathfrak{h}$ by  Lie superalgebra automorphisms. Then one can obtain a simple Lie superalgebra $\mathfrak{g}$ in $\Rep(S_t)$ as $\Ind_{S_{t-j}\times H}^{S_t}(\mathfrak h)$.

Now using this we can state the conjecture similar to Theorem 4.2.5.

\begin{conj}

Any simple Lie superalgebra in $\Rep(S_t)$ is isomorphic to one obtained in the way discussed above. 

Such a simple Lie superalgebra is determined uniquely by the above data up to conjugation of $H$ inside $S_j$ and conjugation of action of $H$ inside of $\Aut(\mathfrak h)$.
\end{conj}

\begin{sproof}

Below is a sketch of the proof for transcendental $t$. In the case of algebraic $t$ the proof might be similar but relies on Conjecture 5.1.5.

The steps of the proof are the same as in Theorem 4.2.5. 

First using the analogue of Proposition 4.0.3 it is easy to see that indeed $\mathfrak{g} = \Ind_{S_{t-j}\times H}^{S_t}(\mathfrak h)$ for some simple Lie superalgebra $\mathfrak{h}$ which is given by the ultraproduct of simple Lie superalgebras $\mathfrak{h}_n$ which remain simple when considered as an object of $Vect$.

The next step is a little bit more vague, since we need to rule out the possibility of almost all $\mathfrak{h}_n$ being of Cartan type with a non-trivial action of $S_{t_n-j}$(analogue of Proposition 4.1.6). This requires some work in the case of algebraic $t$, but is straightforward for transcendental $t$, since then all $\mathfrak{h}_n$ are Lie superalgebras in characteristic $0$ and the dimension of Cartan type superalgebras grows exponentially with $n$ (hence their rank is bounded, hence almost all actions of $S_{t_n-j}$ are trivial).

Next as an analogue of Proposition 4.2.1 it is easy to show that if the dimensions of $\mathfrak{h}_n$ are bounded, then the action of $S_{t_n-j}$ will be trivial for almost all $n$.

Another step is to show that  $S_{t_n-j}\times H$ acts by inner automorphisms of $\mathfrak{h}_n$ for almost all $n$. This is done in a way similar to Proposition 4.2.3 since the groups of outer automorphisms of these superalgebras are some small groups.

The last step is the analogue of Proposition 4.2.4 which is based on the Lemma 2.2.6. To do this one needs to extend Lemma 2.2.6 to cover some more examples relevant for the superalgebras case, which can be done in a similar way.

After all this is done the claim will follow in the same way as in Theorem 4.2.5.

\end{sproof}

\bibliographystyle{alpha}
\bibliography{biblio}

\end{document}